\documentclass{tran-l}

\usepackage{amsmath,amsthm,amssymb,amsxtra,enumerate,bbm,amscd,geometry,pdfsync,tikz,float}
\usepackage[hidelinks]{hyperref}
\usepackage[numbers]{natbib}
\numberwithin{equation}{section}
\allowdisplaybreaks[1]
\usepackage[active]{srcltx}

\renewcommand{\a}{\alpha}
\renewcommand{\d}{\delta}
\newcommand{\f}{\varphi}
\newcommand{\m}{\mu}
\renewcommand{\l}{\lambda}
\newcommand{\s}{\sigma}
\newcommand{\p}{\pi}

\newcommand{\sets}[1]{\mathbb{#1}}
\renewcommand{\AA}{\sets{A}}
\newcommand{\GG}{\sets{G}}
\newcommand{\HH}{\sets{H}}
\newcommand{\MM}{\sets{M}}
\newcommand{\NN}{\sets{N}}
\newcommand{\RR}{\sets{R}}
\newcommand{\bH}{\mathbf{H}}
\newcommand{\nice}[1]{\mathcal{#1}}
\newcommand{\nA}{\nice{A}}
\newcommand{\nB}{\nice{B}}
\newcommand{\nC}{\nice{C}}
\newcommand{\nD}{\nice{D}}
\newcommand{\nF}{\nice{F}}
\newcommand{\nG}{\nice{G}}
\newcommand{\nH}{\nice{H}}
\newcommand{\nI}{\nice{I}}
\newcommand{\nJ}{\nice{J}}
\newcommand{\nN}{\nice{N}}
\newcommand{\nO}{\nice{O}}
\newcommand{\nP}{\nice{P}}
\newcommand{\nR}{\nice{R}}
\newcommand{\nS}{\nice{S}}
\newcommand{\nT}{\nice{T}}
\newcommand{\wt}[1]{\widetilde{#1}}
\newcommand{\arre}{\rightarrow}
\newcommand{\dd}{\mathrm{d}}
\newcommand{\as}{\text{  a.s.}}
\newcommand{\id}{\mathrm{id}}
\newcommand{\conj}{\overline}

\renewcommand{\r}{\varrho}
\newcommand{\cf}{\mathbbm{1}}
\newcommand{\cfa}[2]{\cf_{#1}^{\otimes #2}}

\newcommand{\ten}{\otimes}
\newcommand{\mm}{\mathbbm{m}}
\newtheorem{satz}{Satz}[section]
\newtheorem{ssatz}{Satz}
\newtheorem{thm}[satz]{Theorem}
\newtheorem*{thm*}{Theorem}
\newtheorem{tthm}[ssatz]{Theorem}

\newtheorem{prop}[satz]{Proposition}
\newtheorem{lemma}[satz]{Lemma}

\theoremstyle{definition}
\newtheorem{defin}[satz]{Definition}
\newtheorem{ex}[satz]{Example}
\newtheorem{eex}[ssatz]{Example}
\newtheorem{rem}[satz]{Remark}
\newtheorem*{rem*}{Remark}
\newtheorem{rrem}[ssatz]{Remark}
\newcommand{\sptext}[3]{\hspace{#1 em}\mbox{#2}\hspace{#3 em}}
\newcommand{\cF}{\mathcal{F}}
\newcommand{\cB}{\mathcal{B}}
\newcommand{\cG}{\mathcal{G}}
\newcommand{\cM}{\mathcal{M}}
\newcommand{\cC}{\mathcal{C}}
\newcommand{\cT}{\mathcal{T}}
\newcommand{\cI}{\mathcal{I}}
\newcommand{\cN}{\mathcal{N}}
\newcommand{\cR}{\mathcal{R}}
\newcommand{\E}{\mathbb E}
\newcommand{\R}{\mathbb R}
\newcommand{\N}{\mathbb N}
\newcommand{\G}{\mathbb G}
\renewcommand{\P}{\mathbb P}
\newcommand{\Z}{\mathbb Z}
\newcommand{\mpm}{\MM}
\newcommand{\mpmd}{\mpm^{\mathrm{dyad}}}
\newcommand{\equa}{\begin{eqnarray*}}
\newcommand{\tion}{\end{eqnarray*}}

\begin{document}

\title{Permutation Invariant Functionals of L\'evy Processes}

\author{F. Baumgartner}
\address{Department of Mathematics,
         University of Innsbruck,
         Technikerstra\ss e 19a,
         A-6020 Innsbruck,
         Austria}
\email{florian.baumgartner@uibk.ac.at}

\author{S. Geiss}
\address{Department of Mathematics and Statistics,
         University of Jyv\"askyl\"a, 
         P.O.Box 35 (MaD),
         FI-40014 University of Jyv\"askyl\"a,
         Finland}

\email{stefan.geiss@jyu.fi}

\date{}

\begin{abstract}
We study natural invariance properties of functionals defined on L\'evy processes and show
that they can be described by a simplified structure of the deterministic chaos kernels in It\^o's 
chaos expansion. These
structural properties of the kernels relate intrinsically to a measurability with respect to
invariant $\sigma$-algebras. This makes it possible to apply deterministic functions to 
invariant functionals on  L\'evy processes while keeping
the simplified structure of the kernels. This stability is crucial for applications. Examples are given as well.
\end{abstract}

\maketitle

\tableofcontents


\section*{Introduction}  

In recent years, It\^o's chaos expansion \cite{ito:56} for L\'evy processes was applied to investigate various problems in 
stochastic analysis and stochastic process theory. For example, it was used to investigate quantitative properties of stochastic processes in
continuous time or to prove covariance relations and inequalities, like the Poincar\'e inequality, 
for general Poisson processes, see 
\cite{last:penrose:11,geiss:geiss:laukkarinen:13,briand:labart:14,geiss:steinicke:13,cgeiss:labart:15}. 
Given a L\'evy process $X=(X_t)_{t\in [0,1]}$ 
and letting $L_2(\nF^X):=L_2(\Omega,\cF^X,\P)$, with $\cF^X$ being the completion of $\sigma(X_t: t\in [0,1])$, the chaos expansion 
is  an orthogonal decomposition
\[ L_2(\nF^X) = L_2-\bigoplus_{n=0}^\infty \nH_n, \]
where $F\in L_2(\nF^X)$ is decomposed into
\begin{equation}\label{eqn:chaos_expansion}
 F = \sum_{n=0}^\infty I_n(f_n).
\end{equation}
The functions $f_n \colon ((0,1]\times \R)^n\to \R$ are symmetric and belong to 
$$L_2^n = L_2( ((0,1]\times \R)^n,(\mathcal{B}((0,1])\otimes \mathcal{B}(\R))^{\otimes n}, \mm^{\otimes n}),$$
where $\mm$ is a $\sigma$-finite measure derived from the L\'evy measure $\nu$ of $X$, and in $f_n((t_1,x_1),\ldots,(t_n,x_n))$
the variables $t_1,\ldots,t_n$ represent the time and $x_1,\ldots,x_n$ the state space.
The expressions $I_n(f_n)$ are multiple integrals with respect to a random measure associated with the process 
$(X_t)_{t\in [0,1]}$.
At first glance, the chaos expansion is a perfect tool to describe $L_2$-random variables by
deterministic objects, the chaos kernels. In fact, various stochastic properties of $F$ transfer to or can seen by
means of the kernel functions $f_n$. For example, measurability with respect to $\cF_t^X$, the completion of $\sigma(X_s: s\in [0,t])$, can 
be checked by the support of the $f_n$. Malliavin differentiability or fractional Malliavin differentiability obtained by real interpolation 
can be formulated by moment conditions on the kernels \cite{geiss:geiss:laukkarinen:13}.
Another example can be found in the initial paper of It\^o  \cite{ito:56}, where the chaos expansion was introduced  and
used to investigate the spectral type of operators that are 
induced by a time shift of the underlying process (with the time domain $(-\infty,\infty)$). 
This is a first example to investigate L\'evy-Wiener type spaces by the structure of the chaos kernels in the 
chaos representation.
\smallskip

A general obstacle for the application of the chaos representation is the fact that the chaos kernels depend on an increasing
number of coordinates. As a result, their structure gets involved and computations become difficult or sometimes 
impossible although one can represent the kernel functions in certain 
      situations:
      using difference operators or Malliavin derivatives, kernel representations are obtained in  \cite{geiss:laukkarinen:11} and
       \cite{last:penrose:11} by iterated derivatives where differential properties of $F$ are needed  in the presence of the Brownian motion part
       (or see \cite{yip:stephens:olhede:10}, where powers of increments of the L\'evy process are considered).
An account on involved combinatorial aspects of chaos decompositions and applications, including 
multiplication formulas, can be found in \cite{peccati:taqqu:11}.
\smallskip

The aim of this paper is to restrict the chaos expansion \eqref{eqn:chaos_expansion} to $F\in \bH \subseteq L_2(\cF^X)$, where
$\bH$ is an appropriate closed linear subspace, and to make the expansion applicable in various situations while keeping essential properties 
of the chaos expansion. Applicable means that we reduce the complexity of the kernels by taking into account natural
invariance properties induced by permutation groups, so that the kernels can be handled even if the dimension of the chaos gets large,
in particular, an explicit computation of the kernels will not be needed.
The results are required in recent developments of stochastic analysis and stochastic process theory.
In Example \ref{ex:bsde} below we explain how our results were applied in \cite{geiss:steinicke:13} in the context of BSDEs.
Summarizing, we have two goals: firstly, we want to present results that are needed
in recent developments, secondly we continue the line of research from It\^o \cite{ito:56}.
\smallskip

To explain the invariance properties we have in mind we look at the three elementary examples
\equa
F_1&:= &\Phi_1 \left ( \int_{(0,1/2]}\varphi_t \dd X_t,
                   \int_{(1/2,1]}\varphi_{t-\frac{1}{2}} \dd X_t \right ), \\
F_2 &:= & \Phi_2 \big ([X]_{1/2},[X]_1-[X]_{1/2}\big ), \\
F_3 &:= &\int_0^1 \int_0^t h(t-s) \dd W_s \dd W_t.
\tion
Here, $\varphi\colon[0,1/2]\to \R$ is continuous,
$\Phi_1\colon\R^2\to \R$ symmetric, bounded and measurable,
$\Phi_2\colon\R^2\to \R$ bounded and measurable, but not necessarily symmetric, and
$h\colon[0,1]\to \R$ is bounded and measurable with the symmetry 
$h(1/2-r)=h(1/2+r)$ for $r\in [0,1/2]$. Moreover,
$W$ is the normalized Brownian part of $X$ and $[X]$ denotes the quadratic variation process of $X$,
see \cite[Section II.6]{protter:04}. The time variables of the corresponding kernels
appearing in the second summand of the It\^o chaos expansion of the random variables $F_1,F_2,F_3$ have
symmetries that correspond to the pictures below:
\begin{center}
\begin{tikzpicture}[scale=1.7]
\node at (0,-0.2) {0};
\node at (2,-0.2) {1};
\node at (-0.2,0) {0};
\node at (-0.2,2) {1};
\node at (1,-.5) {Example $F_1$};
\draw[thick] (0,0) rectangle (2,2);
\draw (0,0) to (2,2);
\draw (0,0) rectangle (1,1);
\draw (1,1) rectangle (2,2);
\draw (1,0) to (2,1);
\draw (0,1) to (1,2);
\node at (0.66,0.33){$A$};
\node at (1.66,1.33){$A'$};
\node at (0.33,0.66){};
\node at (1.33,1.66){};
\node at (0.66,1.33){$B'$};
\node at (1.66,0.33){$B$};
\node at (0.33,1.66){};
\node at (1.33,0.66){};
\node at (3,-0.2) {0};
\node at (5,-0.2) {1};
\node at (2.8,0) {0};
\node at (2.8,2) {1};
\node at (4,-.5) {Example $F_2$};
\filldraw[fill=red!20!white] (3,0) rectangle (4,1);
\filldraw[fill=blue!25!white] (4,0) rectangle (5,1);
\filldraw[fill=blue!25!white] (3,1) rectangle (4,2);
\filldraw[fill=green!20!white] (4,1) rectangle (5,2);
\node at (3.5,0.5){$C$};
\node at (3.5,1.5){$D^T$};
\node at (4.5,0.5){$D$};
\node at (4.5,1.5){$E$};
\draw[dashed,very thin,gray] (3,0) to (5,2);
\draw[thick] (3,0) rectangle (5,2);
\node at (6,-0.2) {0};
\node at (8,-0.2) {1};
\node at (5.8,0) {0};
\node at (5.8,2) {1};
\node at (7,-.5) {Example $F_3$};
\draw (6,0) to (8,2);
\foreach \x in {0,0.2,0.4,0.6,0.8,1,1.2,1.4,1.6,1.8,2} {
  \draw[very thin,gray] (6 cm+\x cm ,0 cm) to (8 cm,2 cm-\x cm);
  \draw[very thin,gray] (6 cm,\x cm ) to (8 cm - \x cm ,2 cm);
  }
\draw[thin,blue] (6.4,0) to (8,1.6);
\draw[thin,blue] (6,0.4) to (7.6,2);
\draw[thin,blue] (7.6,0) to (8,.4);
\draw[thin,blue] (6,1.6) to (6.4,2);
\draw[thin,red](7,0) to (8,1);
\draw[thin,red](6,1) to (7,2);
\draw[thick] (6,0) rectangle (8,2);
\end{tikzpicture}
\end{center}
In fact, there are two interacting symmetry groups: the general symmetry in 
$(t_1,x_1)$ and $(t_2,x_2)$, and the symmetries that come from $\Phi_1$, the 
bracket process $([X]_t)_{t\in [0,1]}$ and from $h$.
\smallskip

Example $F_1$ is invariant with respect to an interchange of the 
L\'evy process on $(0,\frac{1}{2}]$ with the process on $(\frac{1}{2},1]$ in the sense that $(X_t)_{t\in [0,1]}$ is replaced by
\[ Y_t:= \begin{cases}
                  X_{t + 1/2} - X_{1/2}   & t\in [0,1/2],\\
                (X_1-X_{1/2}) + X_{t-1/2} & t\in (1/2,1].
                  \end{cases}
          \]
Freezing the state variables $(x_1,x_2)$ of the kernel,
this leads to a symmetry in the time variables
$(t_1,t_2)$, where the areas $A'$ resp.\ $B'$ are copies of $A$ resp.\ $B$ obtained by a {\em shift}.
The remaining parts are determined by the symmetry in $(t_1,x_1)$ and  $(t_2,x_2)$.
\smallskip 

Example $F_2$:
Similarly as described above, the L\'evy process can be exchanged on intervals within $(0,\frac{1}{2}]$ resp.\ $(\frac{1}{2},1]$.
Later we show that this immediately results in the structure
\begin{multline*}
     f_2((t_1,x_1),(t_2,x_2)) \\
   = \cf_C(t_1,t_2) g_C(x_1,x_2) +
     \cf_E(t_1,t_2) g_E(x_1,x_2) +
     \cf_D(t_1,t_2)     g_D(x_1,x_2) +
     \cf_{D^T}(t_1,t_2) g_D(x_2,x_1),
\end{multline*}
where the functions $g_C$ and $g_E$ appearing in the diagonal terms are already symmetric.
\smallskip

Example $F_3$: As we only consider the Brownian motion, there is no
dependence of the kernel on the state variables $x_1$ and $x_2$. Directly, 
from the symmetries of $h$ one checks that the kernel is constant in time
along the lines, whereas on lines with the same color the kernel takes the same values.

In this article, symmetries of this kind are the basis to restrict the chaos expansion to a subspace $\bH$.
Let us list some desired abstract properties of this restricted chaos expansion and describe how the structure
of the paper is derived from their treatment:
\smallskip

\begin{enumerate}
\item [(S)] {\sc Stability:} Given random variables $F_1,\ldots,F_N\in \bH$ and an 
      appropriate bounded random functional $f\colon\Omega\times \R^N\to\R$, such that 
      $f(\cdot,x)\in \bH$ for all  $x\in\R^N$, we would like to guarantee that
      $f(F_1,\ldots,F_N)\in \bH$.
       
\item [(C)] {\sc Consistency:} We consider three different stages of compatibility of $\bH$ with the original chaos decomposition.

\begin{enumerate}[(C1)]
\item Are there closed linear subspaces $\bH_n \subseteq \nH_n$ such that
      \[ \bH = L_2-\bigoplus_{n=0}^\infty \bH_n? \]
\item Can the subspaces $\bH$ and $\bH_n$ be obtained by measurability, i.e.\ are there $\sigma$-algebras $\nA$ and $\nA_n$ such that 
      \[ \bH = L_2(\Omega,\nA,\P) \sptext{1}{and}{1} 
         \bH_n = I_n \Big ( L_2\big(((0,1]\times \R)^n,\nA_n,\mm^{\ten n}\big) \Big )? \] 
\item Can one realize $\nA_n=  \nA_1^{\ten n}$? 
\end{enumerate}
\item [(G)] {\sc Generating property of $\bH_1$}: Does one have that
    \[ \nA = \sigma (F\in \bH_1) \vee \{ A\in \cF^X: \P(A)=0\}? \] 
\end{enumerate}
\smallskip

Chaos expansions based on multiple integrals with respect to a centered independently scattered random measure 
 (also called a centered completely random measure) are usually proved under the condition that the 
control measure is non-atomic (see \cite[Chapter 5.1]{peccati:taqqu:11}). 
Starting with a non-atomic control measure, in Theorem \ref{thm:locally_erogdic_introduction} below 
we also obtain chaos decompositions with control measures that are \emph{not} non-atomic, but
sharing the desirable basic properties (S), (C3), and (G) from above. This could open a way to transfer
properties and results from the non-atomic case to the atomic one.
\smallskip

Before we proceed let us make some detailed comments on the above set of conditions:

\begin{rrem}
\begin{enumerate}
\item Roughly speaking, property (C2) is stronger than the stability property (S): 
      If the map $\omega \mapsto f(\omega,F_1(\omega),\ldots,F_N(\omega))$
      can be defined in  a reasonable way, then the measurability will transfer 
      automatically to the composition and implies $f(F_1,\ldots,F_N)\in \bH$ by (C2).
\item The stability (S) excludes 
      certain choices of $\bH$ such as $\bH=\nH_n$ for some $n\ge 1$.
\item The generating property (G) holds for It\^o's chaos expansion as introduced above, and might be approached 
      by orthogonal polynomials associated to certain L\'evy processes (cf. \cite{nualart:schoutens:00,sole:utzet:08,privault:09}) 
      in order to obtain other cases. For example, it holds for the Hermite expansion of the 
      Gaussian space $(\R^n,\mathcal{B}(\R^n),\gamma_n)$ with $\gamma_n$ being the 
      standard Gaussian measure on $\R^n$, and for functionals $f(N_1)$, where 
      $(N_t)_{t\in [0,1]}$ is a standard Poisson process by exploiting 
      Charlier polynomials \cite[Chapter 6]{privault:09}.
\item In general, condition (C2) does not imply (C3) nor (G): take for $\bH'$ the space 
      of random variables $F$ that are invariant with respect to all dyadic permutations 
      of the underlying L\'evy process, and for $\bH{''}$ the $F$ that are invariant with respect to all 
      dyadic periodic shifts of the underlying L\'evy process.
      We have $\bH' \subseteq \bH{''}$ and in Section \ref{subsec:negative_example} we provide an example 
      that $\bH' \subsetneq \bH{''}$. Because $g\in L_2((0,1])$ is a.s.\ constant if and only if $g$ is a.s.\ invariant with 
      respect to all periodic dyadic shifts, in both cases the first chaos  coincides and equals
      \[ \bH_1'=\bH_1{\hspace*{-.4em}  ''}=\{ I_1(\cf_{(0,1]}g_1) :  \cf_{(0,1]} g_1 \in L_2((0,1]\times \R,\mm) \}, \]
      where $(\cf_{(0,1]} g_1)(t_1,x_1) = g_1(x_1)$.
      Using Theorem \ref{thm:locally_erogdic_introduction}(2) below for $L=1$ and $E_1=(0,1]$ gives properties (C3) and (G) 
      for $\bH'$, so that (C3) and (G) cannot hold for $\bH{''}$ as $\bH' \subsetneq \bH{''}$. This also means, although $g\in L_2((0,1])$ 
      is a.s.\ constant whenever $g$ is invariant with respect to all shifts, this phenomenon does not transfer to It\^o's chaos representation. 
\end{enumerate}

\end{rrem}
\smallskip
Let us explain the structure of the paper along the above listed conditions.
The answer to the problems of consistency (C1) and (C2) (and therefore the problem of stability (S))
is given in Section \ref{sec:invariances_levy_processes} by the following statement, which is part of 
Theorem \ref{thm:main2} and Lemma \ref{lem:equiinv}
(for notation see sections \ref{sec:prelimninaries_levy_processes} -- \ref{sec:invariances_levy_processes}). 

\begin{tthm}\label{thm:equivalences_introduction}
For a group $\GG$ of dyadic measure preserving maps 
\footnote{The concept of a dyadic measure preserving map $g$ is defined in 
Definition \ref{definition:dyadic_g} below.} $g\colon(0,1]\to (0,1]$
and $F\in L_2(\nF^X)$
the following assertions are equivalent:
\begin{enumerate}
\item $F$ is invariant with respect to all $\G$-induced permutations of the underlying 
      L\'evy process $X$.
\item $F\in L_2(\Omega,\nH_\GG,\P)$, where
      \[
          \nH_{\G}:=\sigma \Big (I_n(f_n): f_n \mbox{ symmetric},  f_n = f_n\circ g[n] \mbox{ a.e.},  
                           g\in \G,  n\ge 1 \Big ) 
          \vee \{ A \in \cF^X : \P(A)=0 \} \]
      with $g[n]((t_1,x_1),\ldots,(t_n,x_n)):=((g(t_1),x_1),\ldots,(g(t_n),x_n))$.
\item $F$ has a chaos expansion with symmetric kernels 
      $f_n\in L_2(((0,1]\times\RR)^n,\nI(\G[n]),\mm^{\ten n})$, where 
      $\nI(\G[n])$ is the invariant $\sigma$-algebra of the diagonal group $\G[n]$ on
      $((0,1]\times \R)^n$ induced by $\G$.
\end{enumerate}
\end{tthm}
\smallskip
That means that we have an orthogonal decomposition
\[ L_2(\Omega,\nH_\GG,\P) = L_2-\bigoplus_{n=0}^\infty I_n \Big ( L_2(((0,1]\times\RR)^n,\nI(\G[n]),\mm^{\ten n}) \Big )
\]
where for $n=0$ we take the almost surely constant random variables.
The examples $F_1$, $F_2$, and $F_3$ from the beginning fit into this theorem. In particular,
for the case of shift invariant functionals, which corresponds to the setting in \cite{ito:56}, we get from 
Theorem \ref{thm:equivalences_introduction}:

\begin{eex}
We call $g\colon(0,1]\to (0,1]$ a dyadic periodic shift if there is some 
 integer $d\ge 1$ such that 
\[ g(t) = s_d(t) := \begin{cases}
          t + \frac{1}{2^d}      & : t\in \left (  \frac{k-1}{2^d},\frac{k}{2^d} \right ] \mbox{ and } 1 \le  k < 2^d \\
          t + \frac{1}{2^d} - 1  & : t\in \left (  \frac{2^d-1}{2^d},1 \right ] \\
         \end{cases}.
         \]
A functional $F\in L_2(\cF^X)$ is invariant with respect to all dyadic periodic shifts if and only if 
$F$ is measurable with respect to
\[
          \nH_{\rm shift}:=\sigma \Big (I_n(f_n): f_n \mbox{ symmetric and }  
          f_n = f_n\circ s_d[n]
         ,\,\,
         d,n\ge 1 \Big )
          \vee \{ A \in \cF^X : \P(A)=0 \}, \]
where $s_d[n]$ is introduced in \eqref{eqn:definition:g[n]} below.
\end{eex}
\smallskip

To handle conditions (C3) and (G) we introduce the concept of a {\em locally ergodic set} 
in Definition \ref{definition:locallyergodic_new} below which yields a stronger invariance than 
for instance shift invariance. 
For the following, we let $\nO((0,1])$ be the system of all unions of half-open dyadic intervals (including the empty set). For 
pairwise disjoint and non-empty $E_1,\ldots,E_L\in \nO((0,1])$, we let 
\[ \mathcal{B}((0,1])_E := \mathcal{B} \Big ((0,1]\setminus (E_1 \cup \cdots \cup E_L) \Big ) \vee \sigma (E_1,\ldots,E_L). \]
As part of Theorem \ref{thm:locally_ergodic} below we prove 

\begin{tthm}\label{thm:locally_erogdic_introduction}
Let $E_1,\ldots,E_L\in \nO((0,1])$ be pairwise disjoint and non-empty, and let $F\in L_2(\cF^X)$.
\begin{enumerate}
\item Let $\GG$ be a group of dyadic permutations of $(0,1]$, let
      $E_1,\ldots,E_L$ be {\em locally ergodic} with respect to $\GG$ and let
      the random variable $F$ be invariant with respect to all permutations of
      the underlying L\'evy process $X$ induced by $\GG$. Then there is a representation 
      $F=\sum_{n=0}^\infty I_n(f_n)$ with symmetric kernels 
      $f_n \colon ((0,1]\times \R)^n\to \R$ that are 
      $(\mathcal{B}((0,1])_E \otimes \mathcal{B}(\R))^{\otimes n}$-measurable.
\item The following assertions are equivalent:
    \begin{enumerate}
    \item The random variable $F$ is measurable with respect to
          \[ \sigma \Big (I_1(f_1) : f_1\in L_2^1 \mbox{ is } \mathcal{B}((0,1])_E \otimes \mathcal{B}(\R) \mbox{ measurable}\Big )
             \vee \{ A \in \cF : \P(A)=0 \}. \]
    \item There are symmetric $(\mathcal{B}((0,1])_E \otimes \mathcal{B}(\R))^{\otimes n}$-measurable $f_n\in L_2^n$ with 
          $F = \sum_{n=0}^\infty I_n(f_n)$.
    \item The random variable $F$ is invariant with respect to  all permutations of the underlying L\'evy process induced by
          the groups $\mpmd_{E_1},\dots,\mpmd_{E_L}$, where $\mpmd_{E_l}$ consists of all dyadic permutations that leave
          $E_l^c$ invariant.
    \end{enumerate}
\end{enumerate}
\end{tthm}
\smallskip
Theorem \ref{thm:locally_erogdic_introduction} relies on the results of 
Section \ref{sec:diagonal_groups_and_invariant_sets} that are proved in a wider setting 
and are applicable in other situations as well (like in 
\cite{last:penrose:11}). In Section \ref{sec:examples} we verify the following examples (including the
 introductory example $F_2$) to illustrate 
Theorem \ref{thm:locally_erogdic_introduction}.
\begin{eex}
Given a time net $0\leq r_0<\ldots<r_ L\leq 1$, the examples
\begin{enumerate}
\item $f([X]_{r_1}-[X]_{r_0}, \ldots, [X]_ {r_L}-[X]_{r_{L-1}})$ and
\item $f(S_{r_1}^{r_0}, \ldots, S_{ r_L}^{ r_{L-1}})$
\end{enumerate}
admit invariances with respect to $\mpmd_{(r_{ l-1},r_{ l}]}$. In (1), the process
$([X]_t)_{t\in [0,1]}$ is the {\em quadratic variation} of $(X_t)_{t\in [0,1]}$. 
In part (2) the process $(S_t^a)_{t\in [a,1]}$ is the {\em Dol\'eans-Dade exponential}
$\dd S_t^a = S_{t-}^a \dd X_t$
with initial condition $S_a^a=1$ and chaos representation
     \[ S_t^a = 1+ \sum_{n=1}^\infty I_n\left(\frac{1}{n!}\cfa{(a,t]}{n}\right), \]
where $\cfa{(a,t]}{n}((t_1,x_1),\ldots,(t_n,x_n)):= \cf_{(a,t]}(t_1) \cdots \cf_{(a,t]}(t_n)$ and 
$(X_t)_{t\in [0,1]}$ is assumed to be square-integrable and of mean zero. 
\end{eex}

We conclude with the example mentioned in the beginning:

\begin{eex}\label{ex:bsde}
We describe the situation from \cite{geiss:steinicke:13} where the results of this paper were already applied. 
For this purpose we consider a Backward Stochastic Differential Equation (BSDE)
\[ 
Y_t = F + \int_{(t,1]} f\left (s,Y_s,\int_\R Z_{s,x} h(x) \dd \mu(x)\right ) \dd s 
       - \int_{(t,1]\times \R} Z_{s,x} \dd M(s,x) 
      \mbox{ a.s.}, \quad t\in [0,1],
\]
with $h\in L_2(\R,\mu)$, where the random measure $M$ and the Borel measure $\mu$ (both associated with $X$)
are introduced in Section \ref{sec:prelimninaries_levy_processes} below
and $f$ is an appropriate deterministic generator (for the precise setting see \cite{geiss:steinicke:13}).
Given the initial data $F\in L_2(\cF^X)$ and $f$, one looks for the solution processes 
$(Y_t)_{t\in [0,1]}$ and $(Z_{s,x})_{(s,x)\in [0,1]\times \R}$. To be able to control the 
variation of the BSDE, for example to upper bound $\|Y_t - Y_s \|_2$, the authors in \cite{geiss:steinicke:13} 
assume a time net $0=r_1<\cdots<r_L=1$ such that the kernels $f_n$ in the chaos expansion 
$F = \sum_{n=0}^\infty I_n(f_n)$ are constant on all cuboids 
\[ Q_{l_1,\ldots,l_n}:=(r_{l_1-1},r_{l_1}]\times \cdots \times (r_{l_n-1},r_{l_n}]. \]
One main step in \cite{geiss:steinicke:13} consists in verifying 
in \cite[Theorem 4.2]{geiss:steinicke:13} that the structure of the terminal condition $F$
transfers to the solution processes $Y$ and $Z$. This is done by a Picard iteration, where in  
\cite[Lemma 4.3]{geiss:steinicke:13} Theorem \ref{thm:locally_erogdic_introduction} of this article 
is applied. In Section \ref{sub:sec:BSDE} below we outline the ideas behind this application 
in a more abstract and general way.
\end{eex}

\paragraph{\bf Outline of the paper}
In Section \ref{sec:prelimninaries_levy_processes} we provide some preliminaries for L\'evy processes. The permutation operators acting on functionals of L\'evy processes are introduced in Section \ref{sec:dyadic_permutations_levy-processes}.
The first abstract set of general invariance properties is obtained in Section \ref{sec:invariances_levy_processes}, which is  based on the general concepts recalled in Appendix \ref{app:invariant_sets}.
The main results concerning L\'evy processes are presented in Section \ref{sec:reduced_chaos_expansion}. They are directly derived from the results in the
more general setting given in Section \ref{sec:diagonal_groups_and_invariant_sets}, where we 
consider diagonal groups.
In Section  \ref{sec:examples} we discuss some examples, explain a relation to the chaotic 
expansion of Nualart and Schoutens based on the Teugels martingales, and finally return to Example \ref{ex:bsde}
to discuss an application to backward stochastic differential equations in more detail.
\medskip
 
\paragraph{\bf Some notation} The space of bounded continuous functions on a metric space 
$M$ is denoted by $\nC_b(M)$, the set of positive integers by $\NN$.  Given an $L>0$ and $\xi\in \R$, we shall use the truncation function
$\psi_L(\xi) := \max \{ -L , \min \{ \xi , L \} \}$.


\section{Preliminaries for L\'evy processes}
\label{sec:prelimninaries_levy_processes}

We recall some facts about L\'evy processes, for more information the reader
is referred, for example, to  \cite{applebaum:09} and \cite{sato:99}.
Let $X=(X_t)_{t\in [0,1]}$, $X_t\colon \Omega\to\R$, be a L\'evy process,
where all paths are right-continuous and have  left-limits, $X_0\equiv 0$, and where we assume 
that $(\Omega,\cF,\P)$ is a complete probability space and that $\cF = \sigma (X_t : t\in [0,1])\vee \{ A \in \cF : \P(A)=0 \}$.
To emphasize the minimality of $\cF$ we write $\cF=\cF^X$. There are some places where stochastic integration
is  formally used. Here we assume that as filtration the augmentation of the natural filtration of $X$ is taken.
For $E\in\nB((0,1]\times \RR)$ let
\[ N(E):=\#\{t\in (0,1]\colon (t,\Delta X_t)\in E\}  \]
be the Poisson random measure associated to $X$ with values in $\{ \infty,0,1,2,\ldots \}$. Assuming $B\in \cB(\R)$ with 
$B\cap (-\varepsilon,\varepsilon) = \emptyset$ for some $\varepsilon>0$, we set 
\[ \nu(B) := \E N((0,1]\times B) \]
and by $\varepsilon\to 0$ we obtain  the L\'evy measure $\nu$ on $\cB(\R)$ with
$\nu(\{0\})=0$ and $\int_\R [x^2 \wedge 1] \dd\nu(x)<\infty$. If $\sigma\ge 0$ is the parameter for
the Brownian motion part of $X$, then we define the $\sigma$-finite measures
\equa
   \dd\mu (x) &:= & \s^2 \dd \d_0 (x) + x^2 \dd \nu (x), \\
\dd \mm (t,x) &:= & \dd (\l \otimes \m) (t,x)
\tion
on $\cB(\R)$ and $\cB((0,1]\times \R)$, respectively.
The compensated Poisson random measure is defined by $\wt{N}:=N-\l\ten \nu$ on the
ring of $E\in\nB((0,1]\times \RR)$ with $\mm(E)<\infty$.
For such an $E$  one introduces
\begin{align}\label{levyrandommeasure}
M(E) := \s \left (\int\limits_{E\cap ((0,1]\times \{0\})} \dd W_t\right ) + \lim_{N\arre\infty}
        \!\!\!\int\limits_{E\cap((0,1]\times\{\frac{1}{N}< |x| < N\})} \!\!\!x \dd \wt{N}(t,x),
\end{align}
where $W$ is the Brownian motion part of $X$ and 
the limit is taken in $L_2$. 
To recall It\^o's chaos expansion \cite{ito:56}, we let
\[ L_2^n:=L_2\big(((0,1]\times \RR)^n,\nB(((0,1]\times \RR)^n),\mm^{\ten n}\big) \]
and define for pair-wise disjoint $E_1,\ldots,E_n\in  \nB((0,1]\times \RR)$ with $\mm(E_i)<\infty$
the multiple integral
\[ I_n(f_n):=M(E_1)\cdots M(E_n)
   \sptext{1}{if}{1} 
   f_n((t_1,x_1)\ldots,(t_n,x_n)):=\cf_{E_1}(t_1,x_1)\cdots \cf_{E_n}(t_n,x_n). \]
This extends by linearity and continuity to
$I_n\colon L_2^n \to L_2(\cF^X)$.
For $n\neq m$ the integrals $I_n(f_n)$ and $I_m(f_m)$ are orthogonal for any kernels $f_n$ and 
$f_m$. A kernel $f_n$ is called {\it symmetric} provided that 
\[ f((t_1,x_1),\ldots,(t_n,x_n)) =
   f((t_{\pi(1)},x_{\pi(1)}),\ldots, (t_{\pi(n)},x_{\pi(n)})) \]
for all $(t_1,x_1),\ldots,(t_n,x_n)$ and $\pi\in \nS_n$, where $\nS_n$ is the set of all permutations acting on $\{1,\ldots,n\}$.
The symmetrization of an $f_n\in L_2^n$ is given by
\[    \tilde f_n ((t_1,x_1),\ldots,(t_n,x_n)) 
   := \frac{1}{n!} \sum_{\pi \in \nS_n}   f((t_{\pi(1)},x_{\pi(1)}),\ldots, (t_{\pi(n)},x_{\pi(n)})) \]
and shares the two important properties, $I_n(f_n)= I_n(\tilde f_n)$ a.s. and 
$\|I_n(\tilde f_n)\|_{L_2(\cF^X)}=\sqrt{n!}\|\tilde f_n\|_{L_2^n}$.
By It\^{o}'s orthogonal decomposition \cite{ito:56}, for any $F\in L_2(\nF^X)$ there exist unique symmetric kernels $f_n\in L_2^n$ such that
\[ F = \sum_{n=0}^\infty I_n(f_n)
  \sptext{1}{in}{1} L_2(\cF^X).\]
If $\nH_n := I_n(L_2^n)\subseteq L_2(\cF^X)$ and if $\wt{L}_2^n$ are the (equivalence classes of)
symmetric functions in $L_2^n$, then 
\begin{align*}
\nJ\colon \bigoplus_{n=0}^\infty \wt{L}_2^n &\longrightarrow  L_2(\nF^X)
\cong \bigoplus_{n=0}^\infty \nH_n \\
(f_n)_{n=0}^\infty & \,\,\mapsto \,\,  \sum_{n=0}^\infty I_n(f_n),
\end{align*}
defines an isometric bijection, where
$\bigoplus_{n=0}^\infty \nH_n$ is the $\ell_2$-product and
$\bigoplus_{n=0}^\infty \wt{L}_2^n$ is equipped with the norm
\[ \|(f_0,f_1,\ldots)\|:=\left (\sum_{n=0}^\infty n! \|f_n\|^2 \right )^\frac{1}{2}.\]


\section{Dyadic permutations and L\'evy processes}
\label{sec:dyadic_permutations_levy-processes}

In this section we investigate measure preserving transformations on
$L_2(\cF^X)$ and on the chaos decomposition 
induced by dyadic measure preserving maps
$g\colon (0,1]\to (0,1]$.
The final commutative diagram will be
\begin{align*}
\begin{CD}
L_2(\nF^X) @>{T_g}>> L_2(\nF^X) \\
@A{\nJ}AA @AA{\nJ}A\\
\displaystyle\bigoplus_{n=0}^\infty \wt{L}_2^n @>{S_{g^{-1}}}>>\displaystyle\bigoplus_{n=0}^\infty \wt{L}_2^n
\end{CD}
\end{align*}
and is verified in Theorem \ref{thm:SeqT} below. This diagram transfers Lemma~1 of \cite{ito:56}, where shift operations are
considered, to our setting.
The diagram is based on the fact that by the definition of L\'evy processes,
their increments are exchangeable. Later we investigate
how this exchangeability transfers to certain functionals defined on the process $X$ or more generally, to $L^2(\nF^X)$-random variables.
In order to shorten the presentation, given $0\le a < b \le 1$ and $I:=(a,b]$, we let 
$X_I := X_b - X_a$. The dyadic intervals we denote by
\[ 
   I_k^d := \left ( \frac{k-1}{2^d}, \frac{k}{2^d} \right ]
   \sptext{1}{for}{1}
   d\ge 0 \mbox{ and } k\in \{1,\dots,2^d \}. 
\]

\subsection{Construction of $T_g$}
For an integer $d\geq 0$ we let
\[
     \nH^{X,d}
  := \left\{F\in L_2(\cF^X):
     F=f\big(X_{I_1^d},\dots,X_{I_{2^d}^d}\big),
             f\in \nC_b(\RR^{2^d})\right\}
  \sptext{1}{and}{1}
  \nH^X:=\bigcup_{d\ge 0}\nH^{X,d}.\]
All spaces $\nH^{X,0} \subseteq \nH^{X,1} \subseteq \cdots \subseteq \nH^X$ are linear subspaces  of $L_2(\nF^X)$.
\begin{lemma}\label{lemma:HX0dense}
$\nH^X$ is dense in $L_2(\nF^X)$.
\end{lemma}
\begin{proof}
It is known that  
$\big\{ f(X_{(t_0,t_1]},\dots,X_{(t_{N-1},t_N]}): 0\le t_0<\ldots<t_N\le 1,f\in \nC_b(\RR^N),N\in \NN\big\}$
is dense in $L_2(\nF^X)$, cf. for example \cite{ito:56}. The right-continuity of $(X_t)_{t\in [0,1]}$ yields
our assertion.
\end{proof}
\medskip

\begin{defin}
\label{definition:dyadic_g}
\begin{enumerate}
\item For $d\ge 0$ and $\pi\in \nS_{2^d}$ we define $g_\pi\colon (0,1]\to (0,1]$ by shifting $I_k^d$ onto $I_{\pi(k)}^d$, i.e.
      \[ g_\pi(t):= \frac{\pi(k)}{2^d} - \left (  \frac{k}{2^d} - t \right )
        \sptext{1}{if}{1}
         t\in \left (  \frac{k-1}{2^d},\frac{k}{2^d} \right ]. \]
\item We let
      $\mpmd :=  \{ g_\pi\colon (0,1]\to (0,1]\, : \,\pi\in \nS_{2^d}, d\ge 0 \}$.
\item We say that $g\in\mpmd$ is \emph{represented by} $\p\in\nS_{2^d}$ for some $d\geq0$ if $g=g_\pi$.
\item For $g\in \mpmd$ we let 
      $\deg(g):=\min d$, where the minimum is taken  over all $d\ge 0$ such that $g$ can be 
      represented by some $\p\in\nS_{2^d}$.
\end{enumerate}
\end{defin}
\smallskip

Note that for $d\ge \deg(g)$ the map $g$ can always be represented by some $\p\in\nS_{2^d}$ and that all $g\in \mpmd$ preserve 
the Lebesgue measure.

\begin{defin}
For $g\in \mpmd$, $d\ge \deg(g)$, and $\pi\in \nS_{2^d}$ representing $g$, we define the operator
$ T_g\colon \nH^{X,d} \to \nH^{X,d} $
by
\[
   T_g f\big(X_{I^d_{1}},\ldots, X_{I^d_{2^d}} \big )
         := f\big(X_{I^d_{\pi(1)}},\ldots, X_{I^d_{\pi(2^d)}} \big )
          = f\big(X_{g(I^d_1)},\ldots, X_{g(I^d_{2^d})} \big ),
\]
where $g(I) := \{ g(t) : t\in I\} \subseteq (0,1]$ for $I\subseteq (0,1]$.
\end{defin}
\medskip

\begin{lemma}
\label{lemma:properties_Tg}
\begin{enumerate}
    \item For $d\ge \deg(g)$ the operator $T_g\colon\nH^{X,d}\to \nH^{X,d}$ is well defined.
    \item For $e\ge d \ge \deg(g)$ the operators 
          $T_g\colon\nH^{X,d}\to \nH^{X,d}$ and $T_g:\nH^{X,e}\to \nH^{X,e}$
          are consistent
          in the sense that if $g=g_{\pi_e}=g_{\pi_d}$ with $\pi_e\in \nS_{2^e}$ and $\pi_d \in \nS_{2^d}$, then
          $T_{g_{\pi_e}}|_{\nH^{X,d}\to \nH^{X,d}} = T_{g_{\pi_d}}$.
    \item For $F\in \nH^{X,d}$ with $d\ge \deg(g)$ the random variables $F$ and  $T_g F$ have the same distribution. In particular, $T_g$ is a linear isometry 
	  in $L_2(\nF^X)$.
\end{enumerate}
\end{lemma}

\begin{proof}
(1) Assume that 
$   f_1 \big(X_{I^d_{1}},\ldots, X_{I^d_{2^d}} \big )
  = f_2 \big(X_{I^d_{1}},\ldots, X_{I^d_{2^d}} \big )$ a.s.
Because of the exchangeability of the increments of the L\'evy process,
the permuted vector of increments has the same distribution as the original vector.
Therefore we have that
$    f_1\big(X_{I^d_{\pi(1)}},\ldots, X_{I^d_{\pi(2^d)}} \big )
   = f_2\big(X_{I^d_{\pi(1)}},\ldots, X_{I^d_{\pi(2^d)}} \big )$ a.s.\ 
and the equivalence classes coincide.
Assertion (2) follows from the definition and assertion 
(3) follows by the same distributional argument as in (1).
\end{proof}
\medskip

Because of Lemma \ref{lemma:properties_Tg} we can extend $T_g$ to an $L_2$-isometry 
$T_g:\nH^X\to\nH^X$,
and by Lemma~\ref{lemma:HX0dense}, we obtain an isometry
\[ T_g\colon L_2(\cF^X) \to  L_2(\cF^X). \]
The operator $T_g$ acts on the jump-part of $X$ as follows:
\medskip

\begin{lemma}\label{lemma:permpoisson}
Let $g\in \mpmd$, $N$ be the Poisson random measure associated to $X$, $I=(a,b]$ with $0\le a < b \le 1$
being dyadic, and $E=(c,d)$ with $-\infty<c<d<\infty$ and $0\notin \conj{E}$. Then,
\begin{align}\label{permpoisson}
T_g \int\limits_{I\times E} x \dd N(s,x) &= \int\limits_{g(I)\times E} x \dd N(s,x)\, \as
\end{align}
\end{lemma}

\begin{proof}
The proof follows an idea of \cite{geiss:laukkarinen:11}. We show that for $L\in \NN$ and the truncation $\psi_L(\xi)=\max \{-L,\min\{\xi,L\}\}$
it holds that
\[
   T_g \psi_L\left(\int_{I\times E} x \dd N(s,x)\right) = \psi_L\left(\int_{g(I)\times E} x \dd N(s,x)\right)
   \mbox{ a.s.}
\]
Then the assertion follows from the fact that $\psi_L(F)$ converges in $L_2(\cF^X)$ to $F$ whenever $F\in L_2(\cF^X)$.
For $l\in \N$ with $2/l<d-c$ we define a continuous function $h_l$ such that 
$h_l(x)=x$ on $[c+(1/l),d-(1/l)]$, $h_l(x)=0$ if $x\not\in [c,d]$ and on the remaining parts we take the linear 
interpolation. By construction, 
$\lim_{l\to \infty} h_l(x) = x \cf_E(x)$ and $|h_l(x)| \le  |x| \cf_E(x)$. By definition,
\[
     T_g \psi_L\left( \sum_{\substack{k=1,\ldots,2^n\\a\cdot 2^n<k\leq b\cdot 2^n}}
        h_l \left(X_{I_k^n}\right)\right)
   = \psi_L\left( \sum_{\substack{k=1,\ldots,2^n\\a\cdot 2^n<k\leq b\cdot 2^n}}
        h_l \left(X_{g(I_k^n)}\right)\right)
        \mbox { a.s.},
\]
where we assume that $n\ge \deg(g)\vee n_0$, with $n_0\ge 0$ chosen such that $a$ and $b$ belong to the dyadic grid
with mesh-size $2^{-n_0}$.
Using the fact that for a fixed c\`adl\`ag path $t\to \xi_t=X_t(\omega)$ and for any 
$\varepsilon>0$ one finds a partition 
$0=t_0<\cdots < t_N=1$ such that for all $t_{i-1}\le s < t < t_i$ one has that 
$|\xi_t-\xi_s| \le \varepsilon$ (see \cite[Lemma 1, Chapter 3]{billingsley:99}), one concludes 
by $n\to \infty$ with dominated convergence that 
\[
     T_g \psi_L\left( \sum_{t\in (a,b]} h_l(\Delta X_t) \right)
   = \psi_L\left(  \sum_{t\in (a,b]} h_l(\Delta X_{g(t)}) \right) \mbox{ a.s.}
\]
Letting $l\to \infty$ and using again dominated convergence finally yields
\[
     T_g \psi_L\left( \sum_{t\in (a,b]} \Delta X_t \cf_E(\Delta X_t) \right)
   = \psi_L\left(  \sum_{t\in (a,b]} \Delta X_{g(t)}\cf_E(\Delta X_{g(t)}) \right)
     \mbox{ a.s.} \qedhere
\]
\end{proof}

The Gaussian part of $X$ is handled by the next lemma, which is proved in the Appendix
\ref{section:technical_proofs}.

\begin{lemma}\label{lemma:permbrownian}
Let $g\in \mpmd$ and let $(\sigma B_t)_{t\in [0,1]}$ be the Brownian motion part of $X$, where we assume 
that $\sigma>0$ and that $t\in (0,1]$ is dyadic. Then,
\[  T_g B_t =\int\limits_{g((0,t])}\dd B_s\, \as \]
\end{lemma}
\medskip

\subsection{Construction of $S_g$}
For $g\in\mpmd$ we define the operator
\begin{equation}\label{Sop}
S_g\colon \prod_{n=0}^\infty L_2^n \rightarrow \prod_{n=0}^\infty L_2^n
\sptext{1}{by}{1}
(f_n)_{n=0}^\infty \mapsto \big(S_{g,n}(f_n)\big)_{n=0}^\infty,
\end{equation}
where $S_{g,n}\colon L_2^n \rightarrow L_2^n$ is given by
\[ f_n\big((t_1,x_1),\ldots,(t_n,x_n)\big) \mapsto f_n\big((g(t_1),x_1),\ldots,(g(t_n),x_n)\big). \]
The distributions of $f_n$ and $S_{g,n}f_n$ coincide, so that the operators $S_{g,n}$ and 
$S_g$ are isometries. The next lemma shows that we can restrict ourselves 
to symmetric functionals in $L_2^n$ when investigating $S_g$.
\medskip

\begin{lemma}
\begin{enumerate}
\item For $f_n,h_n\in L_2^n$ with $I_n(f_n)=I_n(h_n)$ one has $I_n (S_{g,n}f_n)=I_n(S_{g,n}h_n)$. 
\item For $f_n\in L_2^n$ one has $I_n(S_{g,n} f_n)=I_n( S_{g,n} \wt{f_n})$.
\end{enumerate}
\end{lemma}
\smallskip

\begin{proof}
(2) follows from (1) by the property $I_n(f_n)=I_n(\wt{f_n})$.
\smallskip

(1) We know that for the symmetrizations $\wt{f_n}$ and $\wt{h_n}$ we have $I_n(f_n)=I_n(h_n)$ a.e.\
if and only if $\wt{f_n}=\wt{h_n}$ a.e. Hence it suffices to show that $\wt{f_n}=\wt{h_n}$ a.e.\ implies 
$\wt{S_{g,n}f_n}=\wt{S_{g,n}h_n}$ a.e. Using the transformation $r=g(s)$, this follows from
\equa
      \left(\wt{S_{g,n}f_n}\right)\big((s_1,x_1),\ldots,(s_n,x_n)\big)
& = & \frac{1}{n!}\sum_{\r \in \nS_n} f_n\big((g (s_{\r (1)}),x_{\r (1)}),\ldots,(g (s_{\r (n)}),
      x_{\r (n)})\big)\\
& = & \frac{1}{n!}\sum_{\r \in \nS_n} f_n\big((r_{\r (1)},x_{\r (1)}),\ldots,(r_{\r (n)},x_{\r (n)})
      \big)\\
& = & \wt{f_n}\big((r_1,x_1),\ldots,(r_n,x_n)\big)\\
& = & \wt{h_n}\big((r_1,x_1),\ldots,(r_n,x_n)\big)\\
& = & \left(\wt{S_{g,n}h_n}\right)\big((s_1,x_1),\ldots,(s_n,x_n)\big)
\tion
for every $((r_1,x_1),\ldots,(r_n,x_n))$ for which $\wt{f_n}$ and $\wt{h_n}$ coincide. This concludes the proof.
\end{proof}

\subsection{The commutative diagram}

\begin{thm}\label{thm:SeqT}
For $g\in\mpmd$ the following diagram is commutative:
\begin{align}\label{commdiag}
\begin{CD}
L_2(\nF^X) @>{T_g}>> L_2(\nF^X) \\
@A{\nJ}AA @AA{\nJ}A\\
\displaystyle\bigoplus_{n=0}^\infty \wt{L}_2^n @>{S_{g^{-1}}}>>\displaystyle\bigoplus_{n=0}^\infty \wt{L}_2^n
\end{CD}
\end{align}
\end{thm}

\begin{proof}
As all linear combinations of
\[   f_n((s_1,x_1),\ldots,(s_n,x_n))
   = \cf_{(a_1,b_1]\times E_1}(s_1,x_1)\cdot\ldots\cdot \cf_{(a_n,b_n]\times E_n}(s_n,x_n), \]
where the $(a_1,b_1],\ldots,(a_n,b_n]$ are dyadic and pair-wise disjoint and the $E_i$ are of form
$E_i=(c_i,d_i)$ with $c_id_i>0$ or $E_i=\{0\}$, are dense in $L_2^n$, and therefore the symmetrizations $\tilde f_n$ are dense in
$\tilde L_2^n$, it suffices to show that
$\nJ S_{g^{-1}} (0,\ldots,0,\tilde f_n,0,\dots) =T_g \nJ (0,\dots,0,\tilde f_n,0,\ldots)$ for all $n\in \NN$.
For this it is sufficient to check that 
$I_n S_{g^{-1},n} f_n = T_g I_n f_n$, which follows from Lemmas 
\ref{lemma:permpoisson},
\ref{lemma:permbrownian}, and
\ref{lemma:composition_and_Tg},
where we use that the sets $g((a_i,b_i])$ are pair-wise disjoint as well. 
\end{proof}

\begin{rem}
There are formulas, the Stroock formulas, to compute the kernels $(f_n)_{n=0}^\infty$ from the
chaos expansion $F=\sum_{n=0}^\infty I_n(f_n)$ as expected values of iterative applications
of differential and difference operators to $F$ (cf.\ \cite[Theorem 3.3]{eddahbi:sole:vives:2005} and 
\cite[Theorem 1.3]{last:penrose:11}). This might be used in the proof of Theorem \ref{thm:SeqT} as well.
Our approach is slightly more direct and self-contained, 
and shows in a way that the invariance properties, we consider, are not intrinsically connected to differentiability.
On the other hand, the Stroock formulas might open the way to use the results of this article
to link structural properties of $F$ to structural properties of the Malliavin derivatives of $F$.
Lemma \ref{lemma:invariance_integrand} below goes in this direction.
\end{rem}


\section{Invariances for L\'evy processes}
\label{sec:invariances_levy_processes}

Throughout this section we let $\GG\subseteq \mpmd$ be a subgroup of the group of dyadic measure preserving maps. 
For $n\in \N$ we derive the group $\GG[n]$ of the measure-preserving $\big((0,1]\times \RR\big)^n$-automorphisms 
\begin{equation}\label{eqn:definition:g[n]}
g[n]: \big((t_1,x_1),\ldots,(t_n,x_n) \big) \mapsto \big ( (g(t_1),x_1),\ldots,(g(t_n),x_n)\big) \sptext{.5}{with}{.5} g\in\GG.
\end{equation}
\smallskip
Now we introduce the main concepts of invariance we are interested in.

\begin{defin}
\begin{enumerate}
\item {\bf $\bH_\G$-invariance}.
      An $F\in L_2(\cF^X)$ is $\G$-invariant if $T_g F=F$ a.s.\ for all $g\in \G$. The set of all $\G$-invariant (equivalence classes of) 
      random variables is denoted by $\bH_\G$.

\item {\bf $\nH_\G$-measurability}.
      A symmetric chaos kernel $f_n\colon((0,1]\times\R)^n\to \R$ is $\G[n]$-invariant if 
      $f_n=f_n \circ g[n]$ a.e.\ for all $g\in \G$. We let
	\[
	\nH_\G:=\s \big( I_n(f_n):f_n \mbox{ is } \G[n]\text{-invariant},\, n\in \NN\big)\vee \nN
	\sptext{.7}{with}{.7}
	\nN:=\{ A \in \cF^X : \P(A)=0\}. \]
\item {\bf $\GG$-invariant chaos expansion}.
       An $F\in L_2(\nF^X)$ has a $\GG$-invariant chaos expansion if all chaos kernels 
      $f_n$ are symmetric and $\G[n]$-invariant. 
\end{enumerate}
\end{defin}
\smallskip
The definition of $\nH_\G$ can understood in the way that we take particular representatives of $I_n(f_n)$ to define the $\sigma$-algebra. By adding 
the null-sets, all representatives become measurable with respect to $\nH_\G$.
The next theorem is the main result of this section:

\begin{thm}\label{thm:main2}
For a group of dyadic measure preserving maps $\G\subseteq \mpmd$ and $F\in L_2(\nF^X)$
the following assertions are equivalent:
\begin{enumerate}
\item\label{it:main2:finv} $F\in\bH_\G$.
\item\label{it:main2:fmeas} $F$ is measurable with respect to $\nH_{\G}$.
\item\label{it:main2:fcompletered} $F$ has a $\GG$-invariant chaos expansion.
\item\label{it:main2:fconstant}  $F$ has symmetric chaos kernels $f_n$, $n\in \NN$, which are constant on the orbits of $\GG[n]$
      on $((0,1]\times\R)^n$.
\end{enumerate}
\end{thm}

\begin{defin}
If $F\in L_2(\cF^X)$ satisfies one of the conditions of Theorem \ref{thm:main2}, then we will say that $F$ is $\GG$-invariant.
\end{defin}

In order to prove Theorem \ref{thm:main2} we start with the following lemma, which is verified in Appendix \ref{section:technical_proofs}.

\begin{lemma}\label{lemma:borelstable}
Let $F_1,\ldots,F_n\in \bH_\GG$ and $\f\colon \RR^n\arre \RR$ be Borel measurable with 
$\f(F_1,\ldots,F_n)\in L_2(\nF^ X)$. Then, $\f(F_1,\ldots,F_n)\in \bH_\G$.
\end{lemma}

\begin{proof}[Proof of Theorem \ref{thm:main2}]
(1) $\Longleftrightarrow$ (3) follows from Theorem \ref{thm:SeqT} and the uniqueness of symmetric kernels in the chaos expansion.

(3) $\Longrightarrow$ (2) follows by definition and the completeness of $(\Omega,\nH_\GG,\P)$.

(4) $\Longrightarrow$ (3) is a consequence of Lemma \ref{lem:equiinv}.

(3) $\Longrightarrow$ (4) First we use Lemma \ref{lemma:invariant_with_measure} to obtain a chaos kernel that is constant on the orbits. 
This new kernel will be  symmetrized which keeps the property that the kernel is constant on the orbits.

(2) $\Longrightarrow$ (1)
As  $\bH_\G$ is a closed subspace of $L_2(\cF^X)$, it is sufficient to check that $\cf_A\in \bH_\G$ for all $A\in \nH_\G$. 
Here it is sufficient to take $A$ such that there exists a sequence $(I_{i_k}(f_{i_k}))_{k\in \NN}$ with $(i_k)_{k\in \NN} \subseteq \NN$
and $\G[i_k]$-invariant kernels $f_{i_k}$ such that
\[ A\in \nG:=\s\big(I_{i_k}(f_{i_k}):k\in \NN\big). \]
By martingale convergence, $\cf_A$ can be approximated in $L_2$ by $\nG_n$-measurable functions, where
\[ \nG_n:=\s\big(I_{i_k}(f_{i_k}): k\in \{1,\ldots,n\}\big). \]
By Doob's factorization lemma (cf.\ \cite[Lemma~II.11.7]{bauer:01}), there are Borel functions $\f_n\colon \RR^n\arre \RR$ such that
\begin{align*}
\E(\cf_A|\nG_n) = \f_n\big(I_{i_1}(f_{i_1}),\ldots,I_{i_n}(f_{i_n})\big) \mbox{ a.s.},
\end{align*}
so that
\[ \lim_n \E | \cf_A - \f_n\big(I_{i_1}(f_{i_1}),\ldots,I_{i_n}(f_{i_n})\big) |^2 = 0. \]
Because of the equivalence (1) $\Longleftrightarrow$ (3) we have that $I_{i_k}(f_{i_k})\in \bH_\G$ for all $k\in \NN$ and
Lemma~\ref{lemma:borelstable} implies that 
\[ \f_n\big(I_{i_1}(f_{i_1}),\ldots,I_{i_n}(f_{i_n}))\in \bH_\G. \]
Because $\bH_\G$ is closed in $L_2(\cF^X)$, we derive that $\cf_A\in \bH_\G$.
\end{proof}


\section{Diagonal groups and locally ergodic sets}
\label{sec:diagonal_groups_and_invariant_sets}

Let $(T,\cT,\tau,(\cT_N)_{N=0}^\infty)$ be a filtered probability space such that 
there are refining partitions 
\[ T= T_{N,1}\cup\cdots\cup T_{N,L_N}, \quad N=0,1,2,\ldots, \]
satisfying the following assumptions:
      \begin{enumerate}
      \item $\cT_N = \sigma ( T_{N,1},\ldots,T_{N,L_N} )$,
      \item $\tau(T_{N,l})>0$ for all $(N,l)$,
      \item $\lim_{N\to \infty} \sup_{l=1,\ldots,L_N} \tau (T_{N,l}) = 0$,
      \item $\cT = \bigvee_{N=0}^\infty \cT_N$.
      \end{enumerate}
We let $\nO(T)$ be the system of countable unions of elements from $\bigcup_{N=0}^\infty \nT_N$ (including the empty 
set). The system forms a topology, in particular a set $G\subseteq T$ is open provided that it is empty or for each $x\in G$ there
is a $T_{N,l}$ with $x\in T_{N,l} \subseteq G$.
\medskip

Finally, we suppose that there is a {\em countable} group $\GG$ of bijective bi-measurable 
$g\colon T\to T$.
\medskip

\begin{defin}\label{definition:locallyergodic_new}
\begin{enumerate}
\item A set $E\subseteq T$ of positive measure is called {\em finite locally ergodic} with respect 
      to $\GG$ provided that there is an $N_E\ge 0$ such that $E\in \cT_{N_E}$ and 
      for all $A:=T_{N,l}\cup T_{N,m} \subseteq E$ with $l\not = m$ and $N\ge N_E$ there is a subgroup 
      $\HH \subseteq \GG$ such that 
      \begin{enumerate}
      \item $g|_{A^c} = \id_{A^c}$ for all $g\in \HH$,
      \item the probability space $(A,\nI(\HH|_A),\tau_A)$ is trivial, i.e.\ contains only sets of measure
            one or zero,  where $\HH|_A$ is the restriction of $\HH$ to $A$ and $\tau_A$ the normalized 
            restriction of $\tau$ to $A$.
      \end{enumerate}
\item A set $E\subseteq T$ is called {\em locally ergodic} with respect to $\GG$ provided that 
      there is a sequence $E^j$ of finite locally ergodic sets with respect to $\GG$ such that
      \[ E^1 \subseteq E^2 \subseteq  \cdots \subseteq E
         \sptext{1}{and}{1} E=\bigcup_{j=1}^\infty E^j. \]
\end{enumerate}
\end{defin}

\begin{rem}
\begin{enumerate}
\item By definition, locally ergodic sets belong to $\nO(T)$.
\item The local ergodicity is stable with respect to passing to open subsets:
      If $\emptyset \not = F\subseteq E$, where $F\in \nO(T)$ and where $E$ is locally ergodic, then $F$ is locally ergodic.
\end{enumerate}
\end{rem}

\begin{proof}
Let us check (2).
By definition, we find finite locally ergodic sets such that
      \[ E^1 \subseteq E^2 \subseteq  \cdots \subseteq E
         \sptext{1}{and}{1} E=\bigcup_{j=1}^\infty E^j. \]
At the same time we find an increasing sequence $F^j\in\nT_{N_j}$, $j\in \NN$, such that
$F=\bigcup_{j=1}^\infty F^j$.
One obtains
\[ F = F \cap E = \bigcup_{j=1}^\infty (F^j\cap E^j) \]
and that $F^j\cap E^j$ is finite locally ergodic because $E^j$ is of this type and $F^j\cap E^j\subseteq E^j$.
\end{proof}

Now we define our {\em diagonal group}: We fix $n\in \NN$ and consider an auxiliary $\sigma$-finite 
measure space $(R,\nR,\rho)$ with $\rho(R)>0$ and the group $\GG[n]$ that consists of all maps 
$g[n]\colon (T\times R)^n \to (T\times R)^n $ given by
\[ \big((t_1,x_1),\ldots,(t_n,x_n)\big) \to \big((g(t_1),x_1),\ldots,(g(t_n),x_n)\big)
   \sptext{1}{with}{1} g\in \GG. \]
\medskip

To formulate our main result, we recall that $\nI(\GG[n])$ denotes the invariant $\sigma$-algebra with
respect to the group $\GG[n]$, see Definition \ref{definition:invariant_sigma_algebra} below. 
For $A\in\cT$ the trace-$\sigma$-algebra on $A$ is denoted by $\cT|_A$.
\medskip

\begin{thm}\label{thm:product_of_locally_ergodic_sets_new}
Let $n\in \NN$, $E_1,\ldots,E_L\in \cT$ be pairwise disjoint and locally ergodic with respect to 
$\GG$,  
\[ \cT_E := \cT|_{T\setminus \left ( \bigcup_{l=1}^L E_l \right )} \vee \s(E_1,\ldots,E_L),
   \sptext{1}{and}{1}
   \cN_n:= \{ A \in (\cT\otimes \cR)^{\otimes n} : (\tau\otimes \rho)^{\otimes n}(A) = 0 \}.\] 
Then $\nI(\GG[n]) \subseteq  (\cT_E\otimes \nR)^{\otimes n}\vee \cN_n$.
\end{thm}
\medskip

\begin{lemma}\label{lemma:small_diagonal_new}
Assume a probability space $(M,\cM,m)$, a decreasing sequence of measurable 
sets $D_0\supseteq D_1 \supseteq \ldots $, a sub-$\sigma$-algebra
$\nI\subseteq \cM$  and
\[ \cG_N := \nI \vee \sigma( A_N\in \cM \mbox{ with } A_N \subseteq D_N). \]
Assume that $m(D_N)\to 0$ as $N\to \infty$. Then
\[ \bigcap_{N=0}^\infty (\cG_N\vee \cN) \subseteq \cI \vee \cN
   \sptext{1}{with}{1}
   \cN:= \{ A\in \cM : m(A)=0\}. \]
\end{lemma}

\begin{proof}
The $\sigma$-algebra $\cG_N$ consists of all
\[ B_N = (I_N\cap D_N^c) \cup A_N  \]
with $A_N \in  \cM$, $A_N \subseteq D_N$ and $I_N\in \nI$.
Therefore $B\in \bigcap_{N=0}^\infty (\cG_N\vee \nN)$ gives $I_N\in \nI$ and $A_N\in \cM$ with
$A_N\subseteq D_N$ 
such that
\[ B_N :=  (I_N\cap D_N^c) \cup A_N
   \sptext{1}{satisfies}{1}
   B_N \Delta B \in \cN
   \sptext{1}{for all}{1} N\ge 0. \]
Defining $C:= \bigcup_{N=0}^\infty (B_N \Delta B)\in \nN$, this implies {\em on} $C^c$ that
\[ B= B_N =  (I_N\cap D_N^c) \cup A_N. \]
Let
\[ I := \bigcup_{N=0}^\infty \bigcap_{k=N}^\infty I_k\in \nI. \]
By construction, 
$I_N=B_N$ on $D_N^c$ and $D_0^c \subseteq D_1^c \subseteq \cdots$.
Therefore,
$I\Delta B \subseteq D_N \cup C$
which implies
$     \P(I\Delta B)
  \le \lim_N \P(D_N) = 0$ and proves the lemma.
\end{proof} 

\begin{proof}[Proof of Theorem \ref{thm:product_of_locally_ergodic_sets_new}]
We assume a partition $R=\bigcup_{j\in J} R_j$ with $\rho(R_j)\in (0,\infty)$. Choosing $\lambda_j\in (0,\infty)$ we can arrange that 
$\rho^0(A) := \sum_{j\in J} \lambda_j \rho(A\cap R_j)$ 
becomes a probability measure which has a strictly positive density with respect to $\rho$. As our statement only concerns null-sets 
we can replace $\rho$ by $\rho^0$, or we can assume w.l.o.g. that $\rho$ itself is a probability measure.
\medskip

I. First we assume that $E_1,\ldots,E_L$ are finite locally ergodic. Let us fix a set $B\in \nI(\GG[n])$ of positive measure.

(a) We observe that $\bigvee_{N\ge 0} (\cT_N \otimes \cR)^{\otimes n} = (\cT \otimes \cR)^{\otimes n}$,
so that martingale convergence yields
\[   \lim_{N\to \infty} f_N
   = \cf_B \quad
     (\tau\otimes\rho)^{\otimes n}\mbox{-a.s.},
     \]
where,  for $(t_1,\ldots,t_n)\in Q^N_{l_1,\ldots,l_n} := T_{N,l_1}\times \cdots \times T_{N,l_n}$,
\[    f_N((t_1,x_1),\ldots,(t_n,x_n)) 
   := 
      \int_{Q^N_{l_1,\ldots,l_n}}
   \cf_B ((s_1,x_1),\ldots,(s_n,x_n)) \frac{\dd\tau(s_1) \cdots \dd\tau(s_n)}{\tau^{\otimes n}(Q^N_{l_1,\ldots,l_n})}. \]
(b) For $N\ge 0$ we let
\[
   \Delta_N :=\hspace{-0.4cm} \bigcup_{\substack{ l_1,\ldots , l_n\in \{1,\ldots,L_N \} \\   
   \text{at least two $l_k$ coincide}}} 
                 \hspace{-0.4cm} Q^N_{l_1,\ldots,l_n}
\]
which is empty for $n=1$.  For $n\ge 2$ the size of $ \Delta_N$ can be upper bounded by
\[ 
   \tau^{\otimes n} (\Delta_N) \le \binom{n}{2} \max_{l=1,\ldots,N} \tau(T_{N,l}) 
   \sptext{1}{so that}{1}
   \lim_N \tau^{\otimes n}(\Delta_{N}) = 0.
\]
Define
\[ \cG_N := (\cT_E\otimes \cR)^{\otimes n} \vee
            \sigma \bigg ( D \times G : D\in \cT^{\otimes n}, D\subseteq \Delta_N,
                    G\in \cR^{\otimes n}\bigg ) \]
with a slight abuse of notation concerning the order of components, 
which gives the $\sigma$-algebra
$\cT_E\otimes \cR$ in the case $n=1$. As $\Delta_0 \supseteq \Delta_1 \supseteq \cdots$ we have 
$\cG_0 \supseteq \cG_1 \supseteq \cdots$.
\medskip

(c) Let $N_0:= \max \{ N_{E_1},\ldots,N_{E_L} \} \ge 0$, where the $N_{E_l}$ are taken from
 Definition~\ref{definition:locallyergodic_new}~(1).
The main observation of the proof is
that $f_N$ is $\cG_N$-measurable for $N\ge N_0$.
By definition, $f_N$ is constant on all cuboids 
$Q^N_{l_1,\ldots,l_n}$. Assume two cuboids
\[ Q^N_{l_1,    \ldots,l_n} 
   \sptext{1}{and}{1}
   Q^N_{m_1,l_2,\ldots,l_n} \]
such that $(l_1,\ldots,l_n)$  are distinct, $(m_1,l_2,\ldots,l_n)$ are distinct, $l_1\not = m_1$, and
that $T_{N,l_1},T_{N,m_1}\subseteq E_l$, where $l\in \{1,\ldots,L\}$ is now  fixed. By assumption, there is a sub-group 
$\HH$ of $\GG$ such that for $A_l:=T_{N,l_1}\cup  T_{N,m_1}$ the probability space $(A_l,\nI(\HH|_{A_l}),\tau_{A_l})$ is 
trivial and $\HH$ acts as an identity 
outside $A_l$. Because $B\in \nI(\GG[n])$ we have that
\[ \cf_B g[n] = \cf_B \sptext{1}{for all}{1} g\in \GG, \]
so that, for all $g\in \HH$,
\[   \cf_B ((g t_1,x_1),(t_2,x_2),\ldots,(t_n,x_n)) 
   = \cf_B (( t_1,x_1),(t_2,x_2),\ldots,(t_n,x_n))  \]
on 
$(A_l\times R)\times
 (T_{N,l_2} \times R) \times \cdots \times (T_{N,l_n}\times R)$.
This implies that the subset $A_l$ of the section of $B$, taken at 
\begin{equation}\label{eqn:section_of_B}
 (x_1,(t_2,x_2),\ldots,(t_n,x_n))\in  R \times
 (T_{N,l_2} \times R) \times \cdots \times (T_{N,l_n}\times R), 
\end{equation} 
is invariant with respect  to $\HH|_{A_l}$ 
and therefore the function 
\[ t_1 \to\cf_B ((t_1,x_1),(t_2,x_2),\ldots,(t_n,x_n)) \]
is almost surely constant on $A_l$ under the condition \eqref{eqn:section_of_B}. Consequently,
\begin{multline*}
     \int_{Q^N_{l_1,\ldots,l_n}} \cf_A (( t_1,x_1),(t_2,x_2),\ldots,(t_n,x_n)) 
      \frac{\dd\tau(t_1)\cdots \dd\tau(t_n)}{\tau^{\otimes n}(Q^N_{l_1,\ldots,l_n})} \\
   = \int_{Q^N_{m_1,\ldots,l_n} } \cf_A (( t_1,x_1),(t_2,x_2),\ldots,(t_n,x_n)) 
   \frac{\dd\tau(t_1)\cdots \dd\tau(t_n)}{\tau^{\otimes n}(Q^N_{m_1,\ldots,l_n})} 
\end{multline*}
for $(x_1,\ldots,x_n)\in R^n$.  
We can repeat the argument, where we replace the exchange of the first component of the 
cuboid by any other component. This implies that $f_N$ is $\cG_N$-measurable.
\medskip

(d) From (c) we immediately get that 
$f_M$ is $\cG_N$-measurable for $M\ge N\ge N_0$. Therefore 
$\cf_B$ is $\cG_N\vee \cN_n$-measurable for all $N\ge N_0$.
Applying Lemma \ref{lemma:small_diagonal_new} we get that
$\cf_B$ is $(\cT_E\otimes \cR)^{\otimes n} \vee \cN_n$-measurable.
\medskip

II. Now we assume general locally ergodic sets $E_1,\ldots,E_L$. By definition, we find monotone sequences 
of finite locally ergodic sets $(E_l^j)_{j=1}^\infty$ with
\[ \bigcup_{j=1}^\infty E^j_l = E_l. \]
We proved in step I that $\nI(\GG[n]) \subseteq  (\cT_{E^j} \otimes \nR)^{\otimes n}\vee \cN_n$ with
\[ \cT_{E^j} := \cT|_{T\setminus \left ( \bigcup_{l=1}^L E_l^j \right )} \vee \s(E_1^j,\ldots,E_L^j),\] 
so that
\[ \nI(\GG[n]) \subseteq \bigcap_{j=1}^\infty \Big ( (\cT_{E^j} \otimes \nR)^{\otimes n}\vee \cN_n \Big ). \]
Observing
\[   (\cT_{E^j} \otimes \nR)^{\otimes n} 
   \subseteq (\cT_{E}   \otimes \nR)^{\otimes n} \vee \sigma(A^j \in (\cT \otimes \nR)^{\otimes n} : 
                                                               A^j \subseteq D^j ) \]
with
\[ D^j:= \left \{ ((t_1,x_1),\ldots,(t_n,x_n)) \in (T\times R)^n : t_k \in \bigcup_{l=1}^L (E_l\setminus E_l^j)
              \mbox{ for some } k\in \{1,\ldots,n \} \right \}, \]
gives that
\[ \nI(\GG[n]) \subseteq  \bigcap_{j=1}^\infty \Big (
(\cT_{E}   \otimes \nR)^{\otimes n} \vee \sigma(A^j \in (\cT \otimes \nR)^{\otimes n} : 
                                                               A^j \subseteq D^j )\vee \cN_n 
\Big ). \]
Finally, because of $D^1 \supseteq D^2 \supseteq \cdots $ and
\[ (\tau\otimes\rho)^{\otimes n}(D^j) \le n \left [ \sum_{l=1}^L \tau(E_l \setminus E_l^j) \right ]
   \to 0 
   \mbox{ as } j\to \infty, \]              
we can again apply Lemma \ref{lemma:small_diagonal_new}.
\end{proof}


\section{Reduced chaos expansions for L\'evy processes}
\label{sec:reduced_chaos_expansion}

In this section we apply the results from Section \ref{sec:diagonal_groups_and_invariant_sets} to L\'evy processes. For this purpose we let 
\begin{enumerate}
\item $  (T,\cT,\tau) := ((0,1], \cB((0,1]),\lambda)$ with $\cT_N =
      \cF_N^{\rm dyad} :=\sigma \left (\left (  \frac{l-1}{2^N},\frac{l}{2^N} \right ] : l=1,\ldots,2^N\right )$,
\item $\mpmd_{E}:=\big\{g\in \mpmd\colon g|_{E^c}=\id_{E^c}\big \}$ for $E\subseteq (0,1]$,
\item $(R,\cR,\rho):=(\R,\nB(\R),\mu)$,
\item and $\cN_n$ be the null-sets in $((0,1]\times \R)^n$ with respect to $(\lambda\otimes\mu)^{\otimes n}$.
\end{enumerate}
\medskip
Let us begin with a prototype of a locally ergodic set.
\medskip

\begin{lemma}\label{lemma:check_ergodicity}
Let $E\in \nO((0,1])$ be non-empty. Then $E$ is locally ergodic with respect to $\mpmd_E$.
\end{lemma}

\begin{proof}
It is enough to show the following: If $A\in \cF_{N_0}^{\rm dyad}$ is a non-empty subset of $E$, then $(A, \nI(\mpmd_{A}|_A),\lambda_A)$ 
is trivial. Take any $B\in \nI(\mpmd_{A}|_A)$. Using the dyadic filtration restricted to $A$, where 
we start with the level $N_0$, we interpret $\cf_B$ as closure of a martingale in $(A,\nB((0,1])|_A,\lambda_A)$ along this 
filtration. By the invariance of $B$, the random variables, that form this martingale, are individually constant. 
Therefore we get a sequence of constants that converge to $\cf_B$ in $L_2(A,\lambda_A)$ and $\lambda_A$-a.s. Hence $\cf_B$ 
is a constant almost surely which implies the statement.
\end{proof}
\medskip

\begin{rem}
One can find groups $\G$ such that for example $E=(0,1]$ is locally ergodic but $\G \subsetneq \mpmd$. Take for 
example all permutations that leave the first interval $(0,2^{-N}]$ invariant on each dyadic level $N$. It would 
be of interest to characterize those sub-groups  $\G \subseteq \mpmd$ such 
that a given $E\in \nO((0,1])$ gets locally ergodic.
\end{rem}

Now we let $\G$ be a group like in Section \ref{sec:invariances_levy_processes}.
The main result is the following simplification of the chaos decomposition: 
\medskip

\begin{thm}\label{thm:locally_ergodic}
For pair-wise disjoint $E_1,\ldots,E_L\in \nO((0,1])$, that are locally ergodic with respect to $\GG$, and 
$F\in L_2(\nF^X)$ consider the following conditions:
\begin{enumerate}
\item $F$ is $\GG$-invariant.
\item One has $F=\sum_{n=0}^\infty I_n(f_n)$ with symmetric, $\G[n]$-invariant, and 
      $(\cB((0,1])_E\otimes \mathcal{B}(\R))^{\otimes n}\vee \cN_n$-measurable $f_n$.
\item One has $F=\sum_{n=0}^\infty I_n(f_n)$ with symmetric $(\cB((0,1])_E\otimes \mathcal{B}(\R))^{\otimes n}$-measurable $f_n$.
\item $F$ is invariant with respect to the group $\HH$ generated by $\mpmd_{E_1},\dots,\mpmd_{E_L}$.
\item $F$ is measurable with respect to
       \[ \nA_E:=\sigma \big (I_1(f_1) : f_1\in L_2^1 \mbox{ is }
          \mathcal{B}((0,1])_E \otimes \mathcal{B}(\R)-\mbox{measurable} \big ) \vee \nN.
       \]
\end{enumerate}
Then it holds that (1) $\Longleftrightarrow$ (2) $\Longrightarrow$ (3) $\Longleftrightarrow$ (4) 
$\Longleftrightarrow$ (5).
Moreover, for $\GG=\HH$ all assertions are equivalent.
\end{thm}

\begin{rem}\label{rem:locally_ergodic_counterexample}
The implication $(4) \Longrightarrow (1)$ of Theorem \ref{thm:locally_ergodic} does not hold in general.
In fact, assume $f:\R^2\to \R\in \nC_b(\R^2)$ such that  
$f(X_\frac{1}{2}-X_0, X_1-X_\frac{1}{2})= f(X_1-X_\frac{1}{2},X_\frac{1}{2}-X_0)$ a.s.\
is not satisfied. Define $F\!:=f(X_\frac{1}{2}-X_0,X_1-X_\frac{1}{2})$, 
 let $\HH\subseteq\mpmd$ be generated by $\mpmd_{(0,\frac{1}{2}]}$ and $\mpmd_{(\frac{1}{2},1]}$,
 and $\GG\subseteq \mpmd$ by $\HH$ and $h\in \mpmd$ exchanging the intervals $(0,\frac12]$ and $(\frac12,1]$. 
 The sets $(0,\frac12]$ and $(\frac12,1]$ are locally ergodic with respect to $\GG$,
 $F$ is invariant with respect to $\HH$, but $F$ is not invariant with respect to $\GG$.
\end{rem}

For the proof of implication $(3)\Rightarrow (5)$ we need product formulas for multiple integrals, 
cf.\ \cite[Theorem 3.5]{lee:shih:04} and \cite[Sections~6.4 and~6.5]{peccati:taqqu:11}. They require 
the definition of a \emph{contraction} of chaos kernels as defined in \eqref{eq:contraction}
according to \cite[Formula (21)]{lee:shih:04} (cf.\ also \cite[Definition~6.2.1]{peccati:taqqu:11}). 
As a preparation, we study the invariance properties of contractions:

\begin{lemma}\label{lemma:lee:shih:invariance}
Let $\HH$ be a group of dyadic permutations of $(0,1]$.
For $n,m\ge 1$, $0\le k \le n\wedge m$, $0\le r \le (n\wedge m)-k$, 
$f\in L_2^n$, and $f'\in L_2^m$ we define 
\[ f\otimes_k^r f' : ((0,1]\times \R)^{n-k-r} \times ((0,1]\times \R)^{m-k-r} \times ((0,1]\times \R)^r
   \to \R
\]
by
\begin{equation}\label{eq:contraction}
\big ( f\otimes_k^r f' \big )(\alpha,\beta,\gamma)
   = \Pi_x(\gamma) 
     \int_{((0,1]\times \R)^k} f(\alpha,\gamma,\rho)
                               f'(\rho,\gamma,\beta) \dd \mm^{\otimes k}(\rho) 
\end{equation}
where $\Pi_x(\gamma)$ is the product of the $x$-coordinates of the vector $\gamma$ and where we assume that
\[ \int_{((0,1]\times \R)^k} |f(\alpha,\gamma,\rho)f'(\rho,\gamma,\beta)| \dd \mm^{\otimes k}(\rho) < \infty \]
for all $(\alpha,\beta,\gamma)\in ((0,1]\times \R)^{n+m-2k-r}$. If $f$ is constant on the orbits of $\HH[n]$ and 
$f'$ is constant on the orbits of $\HH[m]$, then 
$f\otimes_k^r f'$ is constant on the orbits of $\HH[n+m-2k-r]$.
 
\end{lemma}

\begin{proof}
For $g\in \HH$ we simply obtain that
\equa 
&   & \big (f\otimes_k^r f'\big )(g[n-k-r]\alpha,g[m-k-r]\beta,g[r]\gamma) \\
& = & \int_{((0,1]\times \R)^k} f(g[n-k-r]\alpha,g[r]\gamma,\rho) f'(\rho,g[r]\gamma,g[m-k-r]\beta) 
      \dd \mm^{\otimes k}(\rho) \\
& = & \int_{((0,1]\times \R)^k} f(g[n-k-r]\alpha,g[r]\gamma,g[k]\rho)f'(g[k]\rho,g[r]\gamma,g[m-k-r]\beta) \dd \mm^{\otimes k}(\rho) \\
& = & \int_{((0,1]\times \R)^k} f(\alpha,\gamma,\rho)f'(\rho,\gamma,\beta)  \dd \mm^{\otimes k}(\rho) \\
& = & \big (f\otimes_k^r f'\big )(\alpha,\beta,\gamma).
\tion
\end{proof}

For the proof of Theorem \ref{thm:locally_ergodic} we denote by 
$\big ( f\hat\otimes_k^r f' \big )$ the symmetrization of 
$\big ( f\otimes_k^r f' \big )$, and by $f_1\hat\otimes \cdots \hat\otimes f_n$ the symmetrization of
$f_1\otimes \cdots \otimes f_n$.

\begin{proof}[Proof of Theorem \ref{thm:locally_ergodic}]
(2) $\Longrightarrow$ (1) follows from Theorem \ref{thm:main2} and
(1) $\Longrightarrow$ (2) from Theorems  \ref{thm:main2} and \ref{thm:product_of_locally_ergodic_sets_new}.

(2) $\Longrightarrow$ (3) We find an $f'_n=f_n$ a.e.\ that is $(\cB((0,1])_E\otimes \mathcal{B}(\R))^{\otimes n}$-measurable. By symmetrizing 
this $f'_n$, we get a symmetric and $(\cB((0,1])_E\otimes \mathcal{B}(\R))^{\otimes n}$-measurable kernel.

(3) $\Longrightarrow$ (4) follows again from Theorem \ref{thm:main2}.

(4) $\Longrightarrow$ (3) By Theorem  \ref{thm:main2}, we get symmetric kernels that are $\HH[n]$-invariant.
On the other side, Lemma  \ref{lemma:check_ergodicity} yields that $E_1,\ldots,E_L$ are locally ergodic with
respect to $\HH$ so that $\nI(\HH[n]) \subseteq (\cB((0,1])_E\otimes \mathcal{B}(\R))^{\otimes n}\vee \cN_n$ by
Theorem \ref{thm:product_of_locally_ergodic_sets_new}. One can finish as in (2) $\Longrightarrow$ (3).

(5) $\Longrightarrow$ (4) From Lemma~\ref{lemma:check_ergodicity} we know that $E_1,\ldots,E_L$ are locally
ergodic with respect to $\HH$. Next we observe that $I_1(f_1)$ is $\HH$-invariant so that (using the arguments from
(1) $\Longleftrightarrow$ (2) and the a.e.\ uniqueness of $f_1$) one can replace the $f_1$ by $f'_1$ that is $\HH[1]$-invariant. Therefore,
$F$ is $\nH_\HH$-measurable.

(3) $\Longrightarrow$ (5) 
Let $m\ge 1$ and $f_0,\ldots,f_m\in L_2^1$ be step-functions 
based on sets of type $A\times J$ 
with $A \in \{ E_1,\dots,E_L \}$ or $A\subseteq (E_1\cup\cdots \cup E_L)^c$ is a Borel set and $J = (a,b]$ 
or  $J = [-b,-a)$ with $0<a<b<\infty$, or $J = \{ 0 \}$. Then the $f_i$ are constant on the orbits of $\HH[1]$ 
and their integrability assures that we can we apply 
\cite[Theorem 3.5]{lee:shih:04} to get that
\begin{multline*}
I_{m+1} (f_0 \hat\otimes f_1\hat \otimes \cdots \hat \otimes f_m) = \\
      I_1(f_0) I_m(f_1\hat\otimes \cdots \hat\otimes f_m) 
 -  m \big [    I_m     (f_0 \hat \otimes_0^1 (f_1\hat \otimes \cdots \hat \otimes f_m)) 
                  +  I_{m-1} (f_0 \hat \otimes_1^0 (f_1\hat \otimes \cdots \hat \otimes f_m)) 
           \big ].
\end{multline*}
Because of Lemma \ref{lemma:lee:shih:invariance}, all integrands occurring on the the right-hand hand side are constant on the orbits 
of $\HH[1]$, $\HH[m]$,  $\HH[m]$, and $\HH[m-1]$, respectively. This implies that
$I_{m+1} (f_0 \hat\otimes f_1\hat \otimes \cdots \hat \otimes f_m)$
is measurable with respect to
\[  \nH_\HH^{(m)} := \sigma \big (I_n(h_n) : h_n \mbox{ is } \HH[n]-\mbox{invariant}, n\in \{1,\ldots,m\} \big ) 
    \vee \nN. \]
Let $h_{m+1}$ be symmetric and $(\cB((0,1])_E\otimes \mathcal{B}(\R))^{\otimes (m+1)}$-measurable.
It is standard that finite linear combinations of tensor products 
$f_0 \hat\otimes f_1\hat \otimes \cdots \hat \otimes f_m$ of the above form can be used to approximate 
$h_{m+1}$ in $L_2^{m+1}$. 
Using that any $L_2$-convergent sequence contains a sequence that converges almost surely, we get that
$I_{m+1}(h_{m+1})$ is already $\nH_\HH^{(m)}$-measurable.
Induction over $m$ and the identity $\nH_\HH^{(1)} = \nA_E$ yield the implication (3) $\Longrightarrow$ (5).

\medskip

Finally, the equivalence of the assertions in case $\HH=\GG$ is obvious  as $E_1,\ldots,E_L$ are 
locally ergodic with respect to $\HH$ as already used above.
\end{proof}

\begin{rem}
\begin{enumerate}
\item If $E_1,\ldots,E_L$ from Theorem~\ref{thm:locally_ergodic} form a partition of $(0,1]$, then the symmetric kernels $f_n$ 
      in Theorem \ref{thm:locally_ergodic}(3) are constant in the time variables on all cuboids  
      $E_{l_1} \times \cdots \times E_{l_n}$ with $l_1,\ldots,l_n \in \{1,\ldots,L \}$.
\item Given a system of $\mathcal{B}((0,1])_E \otimes \mathcal{B}(\R)$-measurable $f_1^1,f_1^2,\ldots$ such that
         \[ \nA_E = \sigma \big (I_1(f_1^l) : l=1,2,\ldots \big ) \vee \nN,
       \]
      Theorem \ref{thm:locally_ergodic}(5) is equivalent to the fact that we find a functional $\Phi:\R^\N\to \R$, 
      measurable 
      with respect to the Borel $\sigma$-algebra on $\R^\N$ generated by the cylinder sets, such that
      \[ F= \Phi(I_1(f_1^1),I_1(f_1^2), \ldots) \mbox{ a.s.} \]
      This follows from a standard factorization due to Doob (see \cite[Lemma~II.11.7]{bauer:01}).
      For example, for $E=(0,1]$, this leads to representations of $F$ in terms of $B_1$ (the normalized Brownian part if present)
      and $N((0,1],(a,b))$ with $ab>0$.
\end{enumerate}
\end{rem}


\section{Examples and applications}
\label{sec:examples}

\subsection{A negative example: Shift operators}
\label{subsec:negative_example}

First we motivate the need of {\em locally ergodic} sets. We do this by considering
the group generated by shifts, which is inspired by the work of It\^o \cite{ito:56}. Assume that 
\[ F= I_2 (f_2) 
   \sptext{1}{where}{1}
   f_2((s,x),(t,y)):= g_2(|s-t|)h_2(x,y)
\]
with a measurable function $g_2:[0,1]\to \R$ such that $g_2(1/2-s)=g_2(1/2+s)$ for
$s\in [0,1/2]$ and a symmetric Borel function $h_2:\R^2\to \R$ such that 
$f_2 \in \tilde L_2^2$.
It is straightforward to check that $F$ is invariant with respect to all shifts
$s_h:(0,1]\to (0,1]$, $0<h<1$, defined by $s_h(t) := t+h$ if $t+h\le 1$ and $s_h(t):=t+h-1$ if $t+h>1$. 
Obviously, the measure $\mu$ and the functions $g_2$ and $h_2$ can be chosen such that there is no symmetric $\tilde f_2((s,x),(t,y))$ 
not depending on $(s,t)$, but with $f_2=\tilde f_2$ a.e.\ (take for example $\mu$ as the Dirac measure in $1$).

\subsection{Positive examples}
\label{subsec:basicexamples}
Our positive examples are based on Proposition \ref{prop:positive_example} below for which we need the notion
of \emph{weak} $\GG$-invariance:

\begin{defin}
Given a subgroup $\GG \subseteq \mpmd$, we say that an $\cF^X$-measurable random variable $Z:\Omega \to \R$ 
is \emph{weakly} $\GG$-invariant provided that $f(Z)$ is $\GG$-invariant for all $f\in \cC_b(\R)$.
\end{defin}
$\GG$-invariance implies  \emph{weak} $\GG$-invariance by Lemma \ref{lemma:composition_and_Tg}, but the converse does not 
need to be true because of a possibly missing integrability.
To consider our examples, let us fix a sequence of time-points
\[ 0\leq s_1<t_1\leq \ldots \leq s_L<t_L\leq 1 \]
together with the corresponding intervals $\widetilde E_l:= (s_l,t_l]$
for the rest of this section.
Similarly as before, we let
\[    
      \cB((0,1])_{\widetilde E} 
   := \cB\left ( (0,1] \setminus \bigcup_{l=1}^L (s_l,t_l] \right )
      \vee \sigma \left ( (s_1,t_1],\dots,(s_L,t_L] \right ). 
\]
\begin{prop}\label{prop:positive_example} 
Assume $\cF^X$-measurable and weakly $\mpmd_{(s_l,t_l]}$-invariant $Z_1,\ldots,Z_N\colon \Omega\to \R$
for $l\in \{1,\ldots,L\}$, and let $f\colon \R^N\to \R$ be a Borel function with 
$F=f(Z_1,\dots,Z_N)\in L_2(\cF^X)$. Then, there are 
$(\cB((0,1])_{\widetilde E}\otimes \mathcal{B}(\R))^{\otimes n}$-measurable and symmetric chaos kernels $\tilde f_n$ for 
$F$. In particular, they are constant on the cuboids
\[ \prod_{j=1}^n (s_{l_j},t_{l_j}]   \sptext{1}{for}{1}  l_1,\ldots,l_n\in \{1,\ldots,L\}. \] 
\end{prop}

\begin{proof}
The variables $\f(Z_k)$, $k=1,\ldots,N$, are $\mpmd_{(s_l,t_l]}$-invariant, where  $\f(x):=\arctan(x)$. 
Letting  $\psi(y) := \tan(y)$ for $y \in (-\pi/2,\pi/2)$ and $\psi(y) :=0$ otherwise, and using the change of variables
$g(y_1,\ldots,y_N):=f(\psi(y_1),\ldots,\psi(y_N))$,
Lemma~\ref{lemma:borelstable} implies that $F = g(\f(Z_1),\ldots,\f(Z_N))$ is $\mpmd_{(s_l,t_l]}$-invariant. 
The sets $E_l:=(s_l,t_l]$ if $t_l$ is dyadic, and $E_l:=(s_l,t_l)$ otherwise,  belong to $\nO((0,1])$.
According to Lemma~\ref{lemma:check_ergodicity} the set $E_l$ is locally ergodic with respect to $\mpmd_{E_l}$ and therefore 
with respect to the group generated by $\mpmd_{E_1},\ldots,\mpmd_{E_L}$.
Furthermore, observing  that $\mpmd_{E_l} = \mpmd_{(s_l,t_l)} = \mpmd_{(s_l,t_l]}$ if $t_l$ is not dyadic,
Theorem \ref{thm:locally_ergodic} gives the existence of symmetric kernels 
$f_n$ that are $(\cB((0,1])_E\otimes \mathcal{B}(\R))^{\otimes n}$-measurable.
Modifying the kernels on a null set yields the assertion.
\end{proof}

\subsubsection{Dol\'eans-Dade stochastic exponential}
We follow \cite{geiss:geiss:laukkarinen:13} and assume $X$ to be $L_2$-integrable and of mean zero.
For $0\le a \le t \le 1$ we let
\[ S_t^a := 1 + \sum_{n=1}^\infty \frac{I_n(\cf_{(a,t]}^{\otimes n})}{n!}, \]
where we can assume that all paths of $(S_t^a)_{t\in [a,1]}$ are 
c\`adl\`ag for any fixed $a\in [0,1]$. Then we get that
\[ S_t^a = 1 + \int_{(a,t]} S_{u-}^a \dd X_u \mbox{ a.s.}
   \sptext{1}{and}{1}
   S_t = S_t^a S_a \mbox{ a.s.} 
   \sptext{1}{with}{1}
   S_t:= S_t^0.
   \]
Therefore we get from the chaos representation of $S_t^a$:

\begin{lemma}
Each random variable $S_{s_l}^{t_l}$ is $\mpmd_{(s_k,t_k]}$-invariant for  
$k=1,\ldots,L$.
\end{lemma}

One could continue the investigation by using more general  Dol\'eans-Dade exponential formulas (see for example 
\cite[Chapter II, Theorem 37]{protter:04}), which is not done here.

\subsubsection{Limit functionals} Behind the next examples there is common idea formulated in

\begin{defin}\label{definition:C(s,t]}
For $0\le s < t \le 1$ a random variable $Z\colon\Omega\to \R$ belongs to the class
$C(s,t]$ provided that there exists a sequence $0\le N_1 < N_2 < \ldots$ of integers and 
Borel functions $\Phi_k \colon \R^{M_k} \to \R$ such that
\[ Z = \lim_{k\to\infty}Z^{k}:= \lim_{k\to\infty} \Phi_k 
       \left (X_{\frac{a_k}{2^{N_k}}}-X_s,X_{\frac{a_k+1}{2^{N_k}}} - X_{\frac{a_k}{2^{N_k}}},\ldots,
              X_{\frac{b_k}{2^{N_k}}} - X_{\frac{b_k-1}{2^{N_k}}},X_t-  X_{\frac{b_k}{2^{N_k}}}\right )
      \quad \mbox{a.s.,} \]
where $\frac{a_k}{2^{N_k}}$ is the smallest grid point greater than or equal to $s$ and
$\frac{b_k}{2^{N_k}}$ is the largest grid point smaller than or equal to $t$, 
$M_k := b_k - a_k+2$, and the function $\Phi_k$ is symmetric in its arguments 
where the first and last coordinate are excluded.\footnote{Here and in the following 
it is implicitly assumed that the partitions are taken always 
in a way that $\frac{a_k}{2^{N_k}}<\frac{b_k}{2^{N_k}}$ by choosing $N_k$ large enough.}
\end{defin}

\begin{prop}\label{prop:sym(s,t)_chaos_kernels}
Let $Z_1,\ldots,Z_L:\Omega\to \R$ be random variables such that $Z_l$ belongs to the 
class $C(s_l,t_l]$ for $l=1,\ldots,L$, and let $f\colon \R^L\to \R$ be a Borel function with $F:= f(Z_1,\ldots,Z_L) \in L_2(\cF^X)$. 
Then, there are $(\cB((0,1])_{\widetilde E} \otimes \mathcal{B}(\R))^{\otimes n}$-measurable 
and symmetric chaos kernels $\tilde f_n$ for $F$. In particular, they are constant on the cuboids
\[ \prod_{j=1}^n (s_{l_j},t_{l_j}]   \sptext{1}{for}{1}  l_1,\ldots,l_n\in \{1,\ldots,L\}. \]
\end{prop}
\begin{proof}
By Proposition~\ref{prop:positive_example}, it is sufficient to show that $Z_1,\ldots,Z_L$ are weakly 
$\mpmd_{(s_l,t_l]}$-invariant for $l\in \{1,\ldots,L\}$, i.e.\ that $\varphi(Z_m)$ is 
$\mpmd_{(s_l,t_l]}$-invariant for $\varphi\in \nC_b(\R)$ and $m,l\in \{1,\ldots,L\}$.
Let $g\in \mpmd_{(s_l,t_l]}$ (be not the identity).
Then there exists an integer $M\ge 0$ such that 
$g$ acts as a 
permutation of the dyadic intervals of length $2^{-M}$ and as an identity on 
$(s_l,t_l]^c$.
Therefore, there exist integers $0\le a < b \le 2^M$ such that
\[ (s,t]:=\left (\frac{a}{2^{M}},\frac{b}{2^{M}}\right ]\subseteq (s_l,t_l] \]
and $g$ can be described by permuting dyadic intervals on $(s,t]$ of length $2^{-M}$.
By Definition \ref{definition:C(s,t]}, there is an approximation $Z_m = \lim_{k\to\infty} Z_m^k$ a.s. 
By construction, there is a $k_0\ge 1$ such that for all $k\ge k_0$ one has that 
$\varphi(Z_m^k)$ is $T_g$-invariant (here one has to distinguish between the cases $m=l$ and $m\not = l$). By dominated convergence, $\lim_{k\to\infty} \varphi(Z_m^k) = \varphi(Z_m)$ in $L_2(\cF^X)$ so that $\varphi(Z_m)$ is invariant with respect to $T_g$ as well and the proof is complete.
\end{proof}

\begin{ex}\label{example:sym(s,t)}
For $0\le s < t \le 1$ the following random variables belong to the class $C(s,t]$:
\begin{enumerate}
\item $X_t-X_s$.
\item $[X,X]_t-[X,X]_s$, where $[X,X]$ is the quadratic variation process of $X$.
\item $\sup_{r\in (s,t]} |X_r-X_{r-}|$.
\end{enumerate}
\end{ex}

\begin{proof}
(1) is obvious.
(2): Here we first take $\Phi_k (x_1,\ldots,x_{M_k}) := |x_1|^2 + \cdots + |x_{M_k}|^2$ with $N_k=k\geq k_0$, use 
\cite[Chapter II, Theorem 22]{protter:04} to get a sequence that converges in probability, and  
extract a sub-sequence that converges almost surely.\sloppypar
(3) Taking  $\Phi_k (x_1,\ldots,x_{M_k}) := \max \{ |x_1|,\ldots,|x_{M_k}|\}$ and $N_k:=k$ with $k\geq k_0$ and the uniformity result for c\`adl\`ag paths \cite[Chapter~3, Lemma~1]{billingsley:99} yields the assertion.
\end{proof}

\begin{rem}\label{remark:bsde}
Combining  Proposition \ref{prop:sym(s,t)_chaos_kernels} with Example \ref{example:sym(s,t)}(1) yields that the
symmetric chaos kernels $f_n$ of $F=f(X_{t_1}-X_{s_1},\ldots,X_{t_L}-X_{s_L})$ can be chosen to be constant on the cuboids
\[ \prod_{j=1}^n (s_{l_j},t_{l_j}]   \sptext{1}{for}{1}  l_1,\ldots,l_n\in \{1,\ldots,L\}. \]
This was used in \cite{geiss:steinicke:13} in the investigation of variational properties 
of backward stochastic differential equations driven by L\'evy processes.
\end{rem}

\subsection{An application to the chaotic representation property (CRP)}

In this section we show how our results relate to the chaotic representation property
(see \cite{nualart:schoutens:00,sole:utzet:vives:07,peccati:taqqu:11} and for recent results
\cite{ditella:engelbert:16}). Exemplary we consider the chaos expansion due to Nualart and Schoutens, the 
investigation to what extend more general expansions, for example from \cite{ditella:engelbert:16} 
and particularly from their Section 6, 
can be considered is left to future research. For this subsection
we assume that the L\'evy measure satisfies 
$\int_{(-\varepsilon,\varepsilon)^c} \exp(\lambda |x|) \dd\nu(x) < \infty$ for some $\lambda,\varepsilon>0$. 
Then \cite[Theorem 3]{nualart:schoutens:00} (see also \cite[Section 2.2]{sole:utzet:vives:07}) gives an 
orthogonal decomposition 
\[ L_2(\cF^X) = \R\oplus \left ( \bigoplus_{n=1}^\infty \bigoplus_{i_1,\ldots,i_n \ge 1} \mathcal{H}^{(i_1,\ldots,i_n)}\right ), \]
where the spaces $\mathcal{H}^{(i_1,\ldots,i_n)}$ are 
the range of the $n$-fold iterated integrals 
$J^{i_1,\ldots,i_n}_n:L_2 (\Delta_n,\lambda_n)\to \mathcal{H}^{(i_1,\ldots,i_n)}$ 
with respect to the martingales $H^{(i_1)},\ldots,H^{(i_n)}$, obtained by an orthogonalization of the Teugels
martingales, with 
$\Delta_n := \{ 0 < t_1 < \cdots < t_n \le 1 \}$ and $\lambda_n$ being the
Lebesgue measure on $\Delta_n$. That means we have an expansion
\begin{equation}\label{eqn:NS-expansion}
F = \E F +  \sum_{n=1}^\infty \sum_{i_1,\ldots,i_n \ge 1} 
                                      J^{(i_1,\ldots,i_n)}( g_{i_1,\ldots,i_n})
 \sptext{1}{with}{1}
  g_{i_1,\ldots,i_n} \in L_2 (\Delta_n,\lambda_n).
\end{equation}
In the following we explain that, although the spaces $\mathcal{H}^{(i_1,\ldots,i_n)}$ are not invariant with respect 
to the operators $T_g$, the invariance properties still transfer.
\medskip

(a) The spaces  $\mathcal{H}^{(i_1,\ldots,i_n)}$ are not invariant with respect to the operators 
    $T_g$ in general. To see this, let 
    \[ F:=J^{i_1,\ldots,i_n}_n(\cf_{I_1\times \cdots \times I_n}) \]
    with pairwise disjoint dyadic intervals $I_j := ((k_j-1)/2^d,k_j/2^d]\subseteq (0,1]$ with $1\le k_1 < \cdots< k_n \le 2^d$.
    Let  $\pi \in \nS_{2^{d}}$ and $g:=g_{\pi}\in \mpmd$ the corresponding measure preserving map.
    By \cite[Proposition 7]{sole:utzet:vives:07} 
    (the statement is given without proof, a proof can be found 
     in \cite[Proposition 6.9]{ditella:engelbert:16}) and Theorem \ref{thm:SeqT} of this article
    we derive that 
      \[    T_g F
          = J_n^{i_{\sigma(1)},\ldots,i_{\sigma(n)}}
            \big( \cf_{g(I_{\sigma(1)})\times \cdots \times g(I_{\sigma(n)})}
            \big) \]
      where $\sigma \colon \{ 1,\ldots,n\} \to \{ 1,\ldots, n \}$ is the permutation such that 
      the family of intervals $(g(I_{\sigma(j)}))_{j=1}^n$ is ascending.
      Therefore, the chaos $\nH^{(i_1,\ldots,i_n)}$ is not stable with respect to $T_g$ in general.
\medskip
      
(b) The symmetries from the kernels we consider in Theorem  \ref{thm:main2} transfer to the kernels in
      \eqref{eqn:NS-expansion} in the following sense:
      Assume a $\GG$-invariant $F\in L_2(\nF^X)$ 
      and the corresponding symmetric and $\GG$-invariant kernels 
      $(f_n)_{n\in\NN}$ as in Theorem \ref{thm:main2}(4).
      Let $p_1,p_2,\ldots\in L_2(\R,\mu)$ be the orthogonal polynomials  from 
      \cite[Section  2.1]{sole:utzet:vives:07} and $q_n := \| p_n \|_{L_2(\R,\mu)}^2$. 
      For integers $i_1,\ldots,i_n\ge 1$ with $q_{i_1}\cdots q_{i_n}>0$ we define
      $g_{i_1,\ldots,i_n} \in L_2 (\Delta_n,\lambda_n)$ to be
      \begin{equation}\label{eqn:definition_kernels_nualart_schoutens} 
            g_{i_1,\ldots,i_n}(t_1,\ldots,t_n) 
         := n! \int_{\R^n} f_n ((t_1,x_1),\ldots,(t_n,x_n)) 
               \frac{p_{i_1}(x_1)\cdots  p_{i_n}(x_n)}{q_{i_1}\cdots q_{i_n}} \dd\mu(x_1)\cdots \dd\mu(x_n),
      \end{equation}
      where we can assume that the above integral exists point-wise \emph{for all} 
      $(t_1,\ldots,t_n)\in \Delta_n$ (consider
      $A_n := \{ (t_1,\ldots,t_n) \in (0,1]^n : 
                 \int_{\R^n} |f_n ((t_1,x_1),\ldots,(t_n,x_n))|^2 
                              \dd\mu(x_1)\cdots \dd\mu(x_n) < \infty\}$, 
      which is symmetric with $\lambda^{\otimes n}(A_n)=1$, and so that $A_n\times \R^n$ is $\GG[n]$-invariant with a slight abuse of notation concerning the order of components; finally replace
      $f_n$ by $f_n \cf_{A_n\times \R^n}$).
      By \cite[Proposition 7]{sole:utzet:vives:07} 
      we derive that
      \[ \sum_{n=1}^\infty I_n(f_n) = \sum_{n=1}^\infty \sum_{i_1,\ldots,i_n \ge 1\atop q_{i_1}\cdots q_{i_n}>0} 
                                      J^{(i_1,\ldots,i_n)}( g_{i_1,\ldots,i_n}) 
          \mbox{ a.s.} \]
      and the properties 
      of $f_n$ in $(t_1,\ldots,t_n)$ {\em on} $\Delta_n$ directly transfer to the kernels 
      $g_{i_1,\ldots,i_n}(t_1,\ldots,t_n)$ from \eqref{eqn:NS-expansion}. Moreover,
      \eqref{eqn:definition_kernels_nualart_schoutens} and the symmetry of $f_n$ imply that
      \[    g_{i_1,\ldots,i_n}(g(t_{\sigma(1)}),\ldots,g(t_{\sigma(n)})) 
          = g_{i_{\sigma^{-1}(1)},\ldots,i_{\sigma^{-1}(n)}}(t_1,\ldots,t_n) \]
      for $g\in \mpmd$ and the permutation $\sigma\colon \{1,\ldots,n\} \to  \{1,\ldots,n\}$ that guarantees
      $g(t_{\sigma(1)})<\cdots < g(t_{\sigma(n)})$.

\subsection{An application to Backward Stochastic Differential Equations (BSDEs)}
\label{sub:sec:BSDE}
An example of a BSDE driven by a L\'evy process is a formal equation of the form
\begin{equation}\label{eqn:BSDE:Levy}
 Y_t = F + \int_{(t,1]} f\left (s,Y_s,\left (\int_\R Z_{s,x}h_k(x) \dd\mu(x)\right )_{k=1}^N\right ) \dd s 
             - \int_{(t,1]\times \R} Z_{s,x} \dd M(s,x) \mbox{ a.s.,}
       \hspace{1em} t\in [0,1],
\end{equation}
where $h_1,\ldots,h_N \in L_2(\R,\mu)$ for some $N\in \NN$, and further typical assumptions are 
$F \in L_2(\nF^X)$ for the \emph{terminal condition}, and for the 
\emph{generator} $f\colon [0,1]\times \Omega\times \R \times \R^N\to \R$ certain assumptions 
regarding adaptedness, that $\int_{(0,1]} | f(s,0,0) | \dd s \in L_2(\nF^X)$, and that 
$f$ is Lipschitz in $(y,z)$, uniformly in $(s,\omega)$.
The initial data of the BSDE \eqref{eqn:BSDE:Levy} 
are $(F,f)$ and one seeks for the solution processes 
$(Y,Z)$ that consist of an adapted c\`adl\`ag process $Y=(Y_t)_{t\in [0,1]}$ and
a predictable $Z=(Z_{s,x})_{(s,x)\in [0,1] \times \R}$, both satisfying certain
integrability conditions. In order to solve BSDE \eqref{eqn:BSDE:Levy} one might use Picard iterations. The aim 
of this section is to consider this Picard iteration separately and to demonstrate with this how the results and 
concepts of this paper contribute to the BSDE theory.
In particular, we wish to emphasize that we express the properties, we are interested in, in terms of measurability
in parts (2) and (3) of Definition \ref{definition:invariance_BSDE} below which allows us to compose random
objects with these properties and the resulting objects automatically share  the same property (see the proof 
of Theorem \ref{thm:GG-invariant-BSDE} below).
As we detach  the Picard iteration from the remaining BSDE theory we can keep  our assumptions on the 
generator $f$ below minimal so that our results might be applied for different types of BSDEs.
\medskip

{\sc Setting.}
In the following we let $\nF_t^X:= \sigma (X_s : s\in [0,t]) \vee \{ A \in \nF^X:\P(A)=0 \}$ and
obtain a right-continuous filtration $(\nF_t^X)_{t\in [0,1]}$ with $\nF^X=\nF_1^X$. The symbol
$\nP$ denotes the $\sigma$-algebra of predictable events on 
$[0,1]\times \Omega$, i.e., $\nP$ is generated by all adapted path-wise continuous processes. 
A process 
$Z=(Z_{s,x})_{(s,x)\in [0,1]\times \R}$ is called predictable if 
$Z\colon [0,1]\times \Omega \times \R\to \R$ is $\nP\otimes \nB(\R)$-measurable.
Our notion of invariance adapted to BSDEs reads as follows: 

\begin{defin}\label{definition:invariance_BSDE}
Let 
$Y\colon[0,1]\times \Omega\to \R$ be $\nB([0,1])\otimes \nF^X$-measurable,
$Z\colon[0,1]\times \Omega \times \R \to \R$ be $\nB([0,1])\otimes  \nF^X \otimes \nB(\R) $-measurable,
and $\GG$ be a group of dyadic measure preserving maps $\G\subseteq \mpmd$. 
\begin{enumerate}
\item For $t\in [0,1)$ we let 
      $\GG_t := \{ g \in \GG : g(s)=s \text{ for all } s\in(t,1]\}$.
\item We say that $Y$ is $(\GG_t)_{t\in (0,1)}$-invariant provided that 
      for all $t\in (0,1)$ there is a 
      $\mathcal{B}((t,1])\otimes  \nH_{\GG_t}$-measurable
      $Y^{\GG_t}\colon (t,1]\times \Omega \to \R$ with
            $\E \int_{(t,1]} |Y_s- Y_s^{\GG_t}| \dd s=0$.
\item We say that $Z$ is $(\GG_t)_{t\in (0,1)}$-invariant provided that 
      for all $t\in (0,1)$ there is a 
            $\mathcal{B}((t,1])\otimes  \nH_{\GG_t} \otimes \mathcal{B}(\R)$-measurable
            $Z^{\GG_t}\colon (t,1]\times \Omega\times \R \to \R$ with
            $\E \int_{(t,1]\times \R} |Z_{s,x}- Z_{s,x}^{\GG_t}| \dd\mm(s,x)=0$. 
\end{enumerate}
\end{defin}
Of course, part (2) of the definition above is (in a sense) a special case of part (3).
We obtain a family of subgroups $(\GG_t)_{t\in (0,1)}$ of $\GG$ with 
$\GG_s\subseteq \GG_t$ for $0 < s < t < 1$. Moreover, if
$G\in L_2(\nF^X)$ is $\GG_t$-invariant, then $\E(G|\nF_t^X)$ is $\GG_t$-invariant as well since
the conditional expectation corresponds to a restriction of the invariant kernel functions
$f_n$ to $f_n \cf_{(0,t]^n}$ for $n\in \NN$.
\medskip

{\sc Picard scheme.}
We introduce operators $A_{F,f}$ and $B_{F,f}$ in  \eqref{eqn:picard_operators} below
that are used to solve BSDEs by Picard type
iterations. For this purpose we fix the initial data of our BSDE and let the generator 
$f\colon[0,1]\times \Omega\times \R \times \R^N\to \R$ be such that
\begin{enumerate}
\item $f(t,\omega,\cdot,\cdot)\colon \R^{1+N} \to \R$ is continuous for all $(t,\omega)\in [0,1]\times \Omega$,
\item $f\colon[0,t]\times \Omega \times \R \times \R^N\to \R$ is 
      $\nB([0,t]) \otimes \nF_t^X \otimes \nB(\R) \otimes \nB(\R^N)$-measurable for $t\in [0,1]$, and
\item $F \in L_2(\nF^X)$.
\end{enumerate}
One could also investigate the case that $f$ depends in its last coordinate on 
a sequence $(z_k)_{k=1}^\infty$ (for example Fourier coefficients with respect to an orthonormal 
basis $(h_k)_{k=1}^\infty$ in $L_2(\R,\mu)$), but for simplicity we restricted ourselves to finite sequences.

\begin{defin}\label{definition:D_f^h}
Let $Y\colon[0,1]\times\Omega \to \R$ and $Z\colon[0,1]\times \Omega \times \R \to \R$. 
For $N\in \NN$ and $h_1,\ldots,h_N \in L_2(\R,\mu)$ we say $(Y,Z) \in \nD_f^{h_1,\ldots,h_N}$ if
\begin{enumerate}
\item the restriction of $Y$ to $[0,t]\times\Omega \to \R$ is 
      $\nB([0,t])\otimes \nF_t^X$-measurable 
      for $t\in [0,1]$,
\item the restriction of $Z$ to $[0,t]\times \Omega \times \R \to \R$ is $\nB([0,t])\otimes \nF_t^X \otimes \nB(\R)$-measurable 
      for $t\in [0,1]$ and
      $Z_{s,\cdot}(\omega) \in L_2(\R,\mu)$ for all $(s,\omega)\in [0,1]\times \Omega$,
\item $\int_{(0,1]} |f\big(s,Y_s,(\int_\R Z_{s,x}h_k(x) \dd\mu(x))_{k=1}^N\big)| \dd s \in L_2(\nF^X)$.  
\end{enumerate}
\end{defin}

Let $\nS_2$ be the space of adapted c\`adl\`ag processes $Y=(Y_t)_{t\in [0,1]}$ with
$\| Y \|_{\nS_2} := \| \sup_{t\in [0,1]} |Y_t|\|_2 < \infty$ 
and 
\begin{multline*}
   \nP_2:= \Big \{ Z\colon[0,1]\times\Omega\times\R\to \R \mbox{ predictable with }
   \| Z \|_{\nP_2}^2 := \E \int_{(0,1]\times \R} |Z_{s,x}|^2 \dd\mm(s,x) < \infty \\
   \mbox{ and }
   Z_{s,\cdot}(\omega) \in L_2(\R,\mu) \mbox{ for all } (s,\omega)\in [0,1]\times \Omega \Big \}.
\end{multline*}
For any predictable $Z\colon [0,1]\times\Omega\times\R\to \R$ with 
$\E \int_{(0,1]\times \R} |Z_{s,x}|^2 \dd\mm(s,x) < \infty$ one can find a
$Z'\in \nP_2$ with $\E\int_{(0,1]\times \R} |Z_{s,x} - Z'_{s,x}| \dd \mm (s,x) = 0$, so that the last part 
of the definition of $\nP_2$ is not a restriction for us.

\begin{defin}\label{definition:Picard_operators}
We let $A_{F,f}\colon \nD_f^{h_1,\ldots,h_N} \to \nS_2$ and 
$B_{F,f}\colon \nD_f^{h_1,\ldots,h_N} \to \nP_2$ be given by
\begin{equation}\label{eqn:picard_operators}
   A_{F,f} (Y,Z) := (\overline{Y}_t)_{t\in [0,1]} 
   \sptext{1}{and}{1}
   B_{F,f} (Y,Z) :=  (\overline{Z}_{t,x})_{(t,x)\in [0,1]\times \R},
\end{equation}
where  
\[   \overline Y_t
  := \E \left (\left. F + \int_{(t,1]} 
                f\left (s,Y_s,\left (\int_\R Z_{s,x}h_k(x) \dd\mu(x)\right )_{k=1}^N\right ) \dd s \,\right|\, \cF_t^X
      \right ), \]
and the process $\overline Z$ is determined by 
\[       
      \eta
  =   \E\eta  
    + \int_{(0,1]\times \R} \overline Z_{s,x} \dd M(s,x)
    \sptext{1.5}{with}{.75}
    \eta :=  F + \int_{(0,1]} f\left (s,Y_s,\left (\int_\R Z_{s,x}h_k(x) \dd\mu(x)\right )_{k=1}^N\right ) \dd s.
\]
\end{defin}

With Definition \ref{definition:Picard_operators} is meant that $A_{F,f}$ and $B_{F,f}$ map to the corresponding 
equivalence classes (in $\nS_2$ elements of one class are indistinguishable, in $\nP_2$ they
coincide a.e.\ with respect to $\lambda\otimes \P\otimes \mu$) and  we choose one element from each 
equivalence class in applications.
In the BSDE-context iteratives of the operator $A_{F,f}$ usually converge to a 
\emph{generalized non-linear conditional expectation} of the terminal condition $F$ along the generator $f$,
and iteratives of $B_{F,f}$ to a generalized \emph{non-linear gradient} of $F$ along  the generator $f$.

\begin{rem}
\begin{enumerate}
\item  The c\`adl\`ag modification of $\overline Y$ can be obtained by observing
\begin{multline*}
\E \left (\left. F + \int_{(t,1]} 
                f\left (s,Y_s,\left (\int_\R Z_{s,x}h_k(x)\dd \mu(x) \right )_{k=1}^N\right) \dd s \,\right|\, \cF_t^X
      \right ) \\
= \E \left (\left. F + \int_{(0,1]} 
                f\left(s,Y_s,\left(\int_\R Z_{s,x}h_k(x) \dd\mu(x) \right )_{k=1}^N\right ) \dd s \,\right|\, \cF_t^X
      \right ) \\ -
      \int_{(0,t]} 
                f\left (s,Y_s,\left (\int_\R Z_{s,x}h_k(x) \dd \mu(x)\right )_{k=1}^N\right ) \dd s
      \mbox{ a.s.}
\end{multline*}
\item To obtain $\overline{Z}$ (which is $\lambda\otimes\P\otimes\mu$-a.e.\ unique because of It\^o's isometry)
we use the representation property of the random measure $M$, cf. \cite[Chapter 4]{applebaum:09}.
\end{enumerate}
\end{rem}

{\sc Result.} Our contribution is to show that the abstract Picard scheme is invariant with respect 
to $(\GG_t)_{t\in (0,1)}$ provided that the terminal condition $F$ is $\GG$-invariant and the generator 
$f$ is $(\GG_t)_{t\in (0,1)}$-invariant:

\begin{thm}
\label{thm:GG-invariant-BSDE}
Assume $(Y,Z)\in \nD_f^{h_1,\ldots,h_N}$ and $F\in L_2(\nF^X)$ such that
\begin{enumerate}[(i)]
\item $F$ is $\GG$-invariant,
\item for all $t\in (0,1)$ and $(y,(z^k)_{k=1}^N)\in \R^{1+N}$ the restricted generator 
      $f(\cdot,\cdot,y,(z^k)_{k=1}^N):(t,1]\times \Omega\to \R$ is $\nB((t,1])\otimes \nH_{\GG_t}$-measurable,
\item $Y$ and $Z$ are $(\GG_t)_{t\in (0,1)}$-invariant.
\end{enumerate}
Then the following holds:
\begin{enumerate}
\item $A_{F,f}(Y,Z)_t$ is $\GG_t$-invariant for all $t\in (0,1)$.
\item $B_{F,f}(Y,Z)$ is $(\GG_t)_{t\in (0,1)}$-invariant. 
\end{enumerate}
\end{thm}

In order to prove Theorem \ref{thm:GG-invariant-BSDE} we need the following lemma:

\begin{lemma}
\label{lemma:invariance_integrand}
Let $t\in (0,1)$ and
\[ G = \int_{(t,1]\times \R} Z_{s,x} \dd M(s,x) \in \bH_{\GG_t} \]
for some predictable $Z$ with
$\E \int_{(0,1]\times \R} |Z_{s,x}|^2 \dd\mm(s,x)<\infty$.
Then there is a 
$\mathcal{B}((t,1])\otimes \nH_{\GG_t}  \otimes \mathcal{B}(\R)$-measurable
            $Z^{\GG_t}\colon (t,1]\times \Omega\times \R \to \R$ with
$\E \int_{(t,1]\times \R} |Z_{s,x}-Z^{\GG_t}_{s,x}| \dd\mm(s,x)=0$.
\end{lemma}

\begin{proof}
Assume that $G= \sum_{n=1}^\infty I_n(f_n)$, where the $(f_n)_{n\in \NN}$ are symmetric chaos kernels 
that are constant on the orbits of $\GG_t[n]$ on $((0,1]\times \R)^n$ (see Theorem \ref{thm:main2}(4)).
For an integer $N\ge 0$ define 
\[ I_{N,0}^t := (0,t]
   \sptext{1}{and}{1}
   I_{N,k}^t  := \left (t+(1-t)\frac{k-1}{2^N},t+(1-t)\frac{k}{2^N} \right ] \]
for $k=1,\ldots,2^N$, and
\[  J_{N,0}   := \{ 0 \}, \quad
    J_{N,l}   := \left [ \frac{l-1}{2^N},\frac{l}{2^N} \right ), 
    \sptext{1}{and}{.5}
    J_{N,m}   := \left ( \frac{m-1}{2^N},\frac{m}{2^N} \right ] \]
for $l=0,-1,\ldots$ and $m=1,2,\ldots$ The corresponding
$\sigma$-algebras are given by
\[    \nG^t_N 
   := \cB((0,t]) \vee \sigma 
      \left ( I_{N,k}^t : k=1,\ldots,2^N \right )
   \sptext{1}{and}{1}
   \nS_N
   := \sigma \left ( J_{N,l}: l\in \Z \right ). \]
We have that 
$    (\nB((0,1])\otimes \nB(\R))^{\otimes n}
   = \bigvee_{N\ge 0} (\nG^t_N \otimes \nS_N)^{\otimes n}$.
For $P_{N,k}^t \colon L_2((0,1])\to L_2((0,1])$ and $Q_{N,l} \colon L_2(\R,\mu)\to L_2(\R,\mu)$
given by 
$P_{N,0}^t g := \cf_{I_{N,0}^t} g$, 
$Q_{N,0} h   := \cf_{J_{N,0}} h(0)$,
\[ P_{N,k}^t g := \cf_{I_{N,k}^t} \int_{I_{N,k}^t} g(s) \frac{\dd s}{| I_{N,k}^t |},
   \sptext{1}{and}{1}
   Q_{N,l} h   := \cf_{J_{N,l}} \int_{J_{N,l}} h(x) \frac{\dd\mu(x)}{\mu(J_{N,l})},
\]
where $k=1,\ldots,2^N$ and $l\in \Z\setminus \{0\}$, and where
we agree about $ Q_{N,l}:=0$ if $\mu(J_{N,l})=0$,
we define point-wise
\[ \E \left ( f_n \left| (\nG^t_N \otimes \nS_N)^{\otimes n} \right.\right)
   := \sum_{k_1=0}^{2^N} \sum_{l_1=-\infty}^\infty
           \cdots \sum_{k_n=0}^{2^N} \sum_{l_n=-\infty}^\infty
            \left [ A^{t,N}_{k_1,\ldots,k_n \atop l_1,\ldots,l_n} f_n \right ]
      \colon ((0,1]\times\R)^n \to \R \] 
with
\[ A^{t,N}_{k_1,\ldots,k_n \atop l_1,\ldots,l_n} :=  [P_{N,k_1}^t \otimes Q_{N,l_1}] \otimes  \cdots \otimes 
                             [P_{N,k_n}^t \otimes Q_{N,l_n}].  \]
Using the $\sigma$-finiteness of $\mu$, dominated convergence in the sequence space $\ell_1$, 
and martingale convergence, one sees that 
$\E \left ( f_n | (\nG^t_N \otimes \nS_N)^{\otimes n} \right )$ converges to $f_n$ in 
$L_2^n$ as $N\to \infty$.
It is easy to verify (cf.\ \cite{ito:56}) that we can exclude the diagonal terms related to the time interval $(t,1]$ in the 
following sense: We can approximate $f_n$ in $L_2^n$ by finite sums where each summand
is of form
 \[ f_n^0 := \sum_{\pi \in S_n}            
           A^{t,N}_{k_{\pi(1)},\ldots,k_{\pi(n)} \atop l_{\pi(1)},\ldots,l_{\pi(n)}} f_n  \]
for some 
$0 = k_1 = \cdots = k_{n_0} < k_{n_0+1} < \cdots < k_n \le 2^N$ with $n_0 \in \{0,\ldots,n-1\}$ 
and $l_1,\ldots,l_n\in \Z$. The case $n_0=0$ means that $k_1\ge 1$, the case $n_0=n$ can be excluded 
as it would imply $I_n (f_n^0) = 0$ a.s.\ because $\E (G| \nF_t^X)=0$ a.s.\ by our assumption.
For $(s,x)\in (0,1]\times \R$ we set
\[ 
    Z^0_{s,x} 
    :=  n I_{n-1} \left ( 
             f_n^0(\ldots,(t^0,x^0))  
             \cf_{(0,T_{n-1}]^{n-1}}
              \right ) 
             \cf_{I^t_{N,k_n}\times J_{N,l_n}}(s,x)
    \sptext{1}{with}{.5}
    T_{n-1} := t+ (1-t)\frac{k_{n-1}}{2^N} \]
and arbitrary $(t^0,x^0)\in I^t_{N,k_n}\times J_{N,l_n}$.
 By construction,  $Z^0\colon (0,1]\times \Omega \times \R\to \R$ is predictable and the 
restriction   $Z^0\colon (t,1]\times \Omega \times \R\to \R$ is 
$\mathcal{B}((t,1])\otimes \nH_{\GG_t}\otimes \mathcal{B}(\R)$-measurable as 
$ f_n^0 (\ldots,(t^0,x^0))  \cf_{(0,T_{n-1}]^{n-1}}$ is $\GG_t[n-1]$-invariant by construction.
Finally, we have that 
\[   
     I_n(f_n^0                              ) 
   = \int_{(t,1]\times\R} Z_{s,x}^0 \dd M(s,x)
     \mbox{ a.s.} \]
This can be extended to finite linear combinations that approximate $f_n$ in $L_2^n$.
The corresponding restrictions of $Z$-processes form a Cauchy sequence in
$L_2\big((t,1]\times \Omega\times \R, \mathcal{B}((t,1])\otimes \nH_{\GG_t} \otimes \mathcal{B}(\R), 
         \lambda\otimes\P\otimes \mu\big)$
and we find a $\mathcal{B}((t,1])\otimes \nH_{\GG_t}\otimes \mathcal{B}(\R)$-measurable limit as well.
\end{proof}
\bigskip

\begin{proof}[Proof of Theorem \ref{thm:GG-invariant-BSDE}]
(1) We fix $t\in (0,1)$, replace $(Y_s)_{s\in (t,1]}$ by $(Y_s^{\GG_t})_{s\in (t,1]}$,
$(Z_{s,x})_{(s,x)\in (t,1]\times \R}$ 
by $(Z_{s,x}^{\GG_t})_{(s,x)\in (t,1]\times \R}$, and have
\[   \overline Y_t 
   = \E \left (\left. F + \int_{(t,1]} f\left (s,Y_s^{\GG_t},\left (\int_\R Z_{s,x}^{\GG_t}h_k(x) \dd\mu(x)\right )_{k=1}^N\right ) \dd s\, \right|\, \cF_t \right )
   \mbox{ a.s.} \]
as well. By Fubini's theorem the processes 
$\int_\R Z_{\cdot,x}^{\GG_t}h_k(x) \dd\mu(x) \colon (t,1]\times \Omega\to \R$
are $\nB((t,1])\otimes \nH_{\GG_t}$-measurable for $k=1,\ldots,N$. Therefore,
$\int_{(t,1]} f\left (s,Y_s^{\GG_t},\left (\int_\R Z_{s,x}^{\GG_t}h_k(x) \dd\mu(x) \right )_{k=1}^N\right ) \dd s$ is
$\nH_{\GG_t}$-measurable and finally
$\overline Y_t$ is $\GG_t$-invariant.
\medskip

(2) It follows from the definition of $(\overline Y,\overline Z)$ that
\[
     \int_{(t,1]\times\R} \overline Z_{s,x} \dd M(s,x)
   = F  + \int_{(t,1]} f\left (s,Y_s,\left (\int_\R Z_{s,x} h_k(x) \dd\mu(x) \right )_{k=1}^N\right ) \dd s
 - \overline Y_t \mbox{ a.s.}
\]
Again we replace $(Y_s)_{s\in (t,1]}$ by $(Y_s^{\GG_t})_{s\in (t,1]}$,
$(Z_{s,x})_{(s,x)\in (t,1]\times \R}$  by $(Z_{s,x}^{\GG_t})_{(s,x)\in (t,1] \times \R}$,
and deduce by step (1) that $\int_{(t,1]} \int_\R \overline Z_{s,x} \dd M(s,x)$ is $\GG_t$-invariant.
We conclude by Lemma \ref{lemma:invariance_integrand}.
\end{proof}

{\sc Application.} 
We fix the data $(F,f)$ of a BSDE  such that 
\begin{enumerate}[($P_1$)] 
\item $(F,f)$  satisfy conditions (1), (2), and (3) 
      listed before Definition \ref{definition:D_f^h},
\end{enumerate}
and a Picard scheme such that the following is satisfied:
\begin{enumerate}[($P_1$)]
 \setcounter{enumi}{1}
\item $(Y^k,Z^k)\in \nD_f^{h_1,\ldots,h_N}$ for $k=0,1,2,\ldots$
\item $Y^{k+1} = A_{F,f}(Y^k,Z^k)$ and  $Z^{k+1} = B_{F,f}(Y^k,Z^k)$ for $k=0,1,2,\ldots$
\item There is a sub-sequence $0 \le k_1 < k_2 < \cdots $ such that 
      $Y^{k_l} \to Y$ $\lambda\otimes\P$-a.e.\ and $Z^{k_l} \to Z$ $\lambda\otimes\P\otimes\mu$-a.e.\
      as $l\to \infty$, where $Y$ is adapted and c\`adl\`ag and $Z\in \nP_2$.
\item $\int_{(0,1]} |f\big(s,Y_s,(\int_\R Z_{s,x}h_k(x) \dd\mu(x))_{k=1}^N\big)| \dd s \in L_2(\nF^X)$.
\item The pair $(Y,Z)$ satisfies BSDE \eqref{eqn:BSDE:Levy}.
\end{enumerate}
If the initial data $(F,f)$ of the BSDE satisfy the invariance conditions (i) and (ii) of 
Theorem \ref{thm:GG-invariant-BSDE} and if the initial processes $(Y^0,Z^0)$ in the Picard scheme
are $(\GG_t)_{t\in (0,1)}$-invariant, then 
$Y_t$ is $\GG_t$-invariant for all $t\in [0,1]$ and the $Z$-process is $(\GG_t)_{t\in (0,1)}$-invariant.
These invariances can be verified as follows:
\medskip

(a) The $Z$-process: By Theorem \ref{thm:GG-invariant-BSDE} we know that all $Z^k$, $k\in \NN$, 
are $(\GG_t)_{t\in (0,1)}$-invariant.
We fix some $t\in (0,1)$ and get that the restrictions
$Z^{k_l}\colon(t,1]\times \Omega\times \R\to \R$ converge a.e. Replacing the restrictions by 
$(Z^{k_l})^{\GG_t}\colon(t,1]\times \Omega\times \R\to \R$ it follows that 
$(Z^{k_l})^{\GG_t}\colon(t,1]\times \Omega\times \R\to \R$ converge to $Z\colon(t,1]\times \Omega\times \R\to \R$ a.e.
Therefore we find a $Z^{\GG_t}\colon(t,1]\times \Omega\times \R\to \R$ which is 
$\nB((t,1])\otimes \nH_{\GG_t}\otimes \nB(\R)$-measurable which coincides a.e.\ with the restriction
of $Z$.
\smallskip

(b) The $Y$-process: 
By Theorem \ref{thm:GG-invariant-BSDE} and our assumptions, $Y_t^k$ is $\GG_t$-invariant for $k\ge 1$ and 
$t\in (0,1)$. As $Y_1^k=F$ a.s., this extends to $t=1$.
By the c\`adl\`ag property of the $Y^k$ the restrictions 
$Y^{k_l}\colon(t,1]\times \Omega\to \R$ are $\nB((t,1])\otimes \nH_{\GG_t}$-measurable.
Therefore, by the arguments used in (a) we get that there is a $Y^{\GG_t}\colon(t,1]\times \Omega \to \R$ which
is $\nB((t,1])\otimes \nH_{\GG_t}$-measurable with $\E \int_{(t,1]} |Y_s - Y_s^{\GG_t}| \dd s =0$.
Consequently, $Y$ and $Z$ are $(\GG_t)_{t\in (0,1)}$-invariant.
Applying Theorem \ref{thm:GG-invariant-BSDE} to the BSDE \eqref{eqn:BSDE:Levy} gives that 
$Y_t$ is $\GG_t$-invariant for $t\in (0,1)$. 
\medskip

\begin{rem}
For $N=1$ the above conditions ($P_1$)-($P_6$) are fulfilled for example in \cite[Lemma 2.4]{tang:li:94} 
(see also \cite[Theorem 2.2 and pp. 34-35]{geiss:steinicke:1212.3420v3}) in the case of Lipschitz BSDEs.
Moreover, in \cite{geiss:steinicke:1212.3420v3,geiss:steinicke:13} 
$(F,f)$ and $(Y^0,Z^0)$ have the suitable invariance properties with respect to $\GG$ generated 
by $\mpmd_{(r_{l-1},r_l]}$, $l=1,\ldots,L$, for some fixed partition $0=r_0 < r_1 < \cdots < r_L=1$.
In particular, the generator has the form
\[ f(s,\omega,y,z) := f_0(s,X_s(\omega),y,z) \]
where $f_0\colon [0,1]\times \R^3\to \R$ is Lipschitz in $(x,y,z)$ uniformly in $s$ and continuous in $(s,x,y,z)$. 
Therefore, $(f_0(s,X_s,y,z))_{s\in [0,1]}$ is a c\`adl\`ag process and for all $t\in (0,1)$ and $(y,z)\in \R^2$
the restricted generator $f(\cdot,\cdot,y,z):(t,1]\times \Omega\to \R$ is 
$\nB((t,1])\otimes \nH_{\GG_t}$-measurable.
\end{rem}


\begin{appendix}
\section{Invariant sets}
\label{app:invariant_sets}

We recall concepts related to classical ergodic theory (see \cite[Chapter 10]{kallenberg:02} or 
\cite[Chapter V]{shiryaev:96}) and adapt them to our setting. 
The proofs of Lemmas~\ref{lem:equiinv} and~\ref{lemma:quasi_atoms} are omitted (for convenience 
they can be found in \cite{baumgartner:geiss:14:arxiv}) as the assertions are standard.

We assume a measurable space $(S,\Sigma)$ and a group 
$\AA$ of automorphisms of $S$, i.e.\ bijective bi-measurable functions $T\colon\,S\rightarrow S$.

\begin{defin}\label{definition:invariant_sigma_algebra}
The {\em invariant} $\s$-algebra w.r.t.\ $\AA$ is given by
      \[ \nI(\AA):=\{B\in\Sigma\,\colon\,B=T^{-1}(B)\text{ for all } T\in \AA\}. \]
\end{defin}
\medskip

\begin{lemma}\label{lem:equiinv}
For a function $\xi\colon\,S\arre \RR$ the following assertions are equivalent:
\begin{enumerate}
\item\label{it:nImeas} $\xi$ is $\nI(\AA)$-measurable.
\item\label{it:sigmameasconst} $\xi$ is $\Sigma$-measurable and constant on the orbits
      $\{ Ts : T\in \AA \}$, $s\in S$.
\item\label{it:sigTa} $\xi$ is $\Sigma$-measurable and $\xi\circ T=\xi$ for all $T\in\AA$.
\end{enumerate}
\end{lemma}

Let $(S,\Sigma,\gamma)$ be a $\sigma$-finite measure space with $\gamma(S)>0$, $\AA$ be a group of automorphisms acting on $S$, and
\[ \conj{\nI(\AA)}:=\nI(\AA)\vee \nN
   \sptext{1}{where}{1}
   \nN:=\{B\in \Sigma\,\colon\,\gamma(B)=0\}. \]
The equivalence class of $\xi$ w.r.t.\ to the $\gamma$-a.e.-equivalence is denoted by $[\xi]$.
\begin{defin}
The measure $\gamma$ is called \emph{quasi-invariant} w.r.t.\ $\AA$, if $\gamma(T^{-1}B)=0$ for all 
$B\in\nN$ and $T\in\AA$.
\end{defin}

\begin{lemma}\label{lemma:invariant_with_measure}
Let $(S,\Sigma,\gamma)$ be a $\sigma$-finite measure space with $\gamma(S)>0$ and $\AA$ be a group of automorphisms acting on $S$.
Then one has the following assertions:
\begin{enumerate}
\item The operation $[\xi]\circ T := [\xi\circ T]$ is well-defined for all $T\in \AA$ and 
      $\Sigma$-measurable $\xi\colon S\to \R$ if and only if
      $\gamma$  is quasi-invariant w.r.t.\ $\AA$.
\item Let $\gamma$ be quasi-invariant w.r.t.\ $\AA$ and $\AA$ be countable. Then
      $[\xi]\circ T = [\xi]$ for all $T\in\AA$ if and only if 
      $\xi$ is $\conj{\nI(\AA)}$-measurable.
\end{enumerate}
\end{lemma}
\begin{proof}
(1) Assume that  $\gamma$  is quasi-invariant and that  $\xi\colon S\to \R$ is $\Sigma$-measurable.
Then for $\xi_1,\xi_2\in[\xi]$ it holds that $\gamma(\xi_1\neq\xi_2)=0$ and the set 
\[ \{s\,\colon\,\xi_1(Ts) \neq \xi_2(Ts)\}=\{T^{-1}t\,\colon\,\xi_1(t)\neq\xi_2(t)\} \]
has measure zero as well so that the operator $[\xi]\mapsto [\xi]\circ T$ is well-defined. 
For the other implication let $B$ be of measure zero and $\xi:= \cf_B$ so that $[\xi]=0$. 
By assumption, $[\xi\circ T]=0$ and 
\[ 0=\gamma(\{s:\cf_B(Ts)\not =0\})=\gamma(T^{-1}(B)). \]
(2) If there exists an $\nI(\AA)$-measurable $\xi_0\in[\xi]$ it is obvious that the equivalence class is 
 invariant by (1) and Lemma~\ref{lem:equiinv}. Conversely, let $[\xi\circ T]=[\xi]$ for all $T\in\AA$. Define 
\[ S_0:=\big\{s\in S\,\colon\, \xi\circ T(s)=\xi(s) \mbox { for all } T\in \AA \big\} 
       = \bigcap_{T\in\AA}\{s\in S\,\colon\,\xi\circ T(s)=\xi(s)\}, \]
which is a set of co-measure zero because $\AA$ is countable. 
It is standard to check that $S_0\in \nI(\AA)$.
Setting $\xi_0(s):=\xi(s) \cf_{S_0}(s)$,
we obtain from Lemma \ref{lem:equiinv} that $\xi_0$ is $\nI(\AA)$-measurable and $\gamma$-a.e.\ equal to $\xi$.
\end{proof}

\begin{defin}\label{definition:quasi-atom}
Let $(S,\cI,\gamma)$ be a $\sigma$-finite measure space with $\gamma(S)>0$.
A set $A\in \cI$ with $\gamma(A)>0$ 
is called {\em quasi-atom} provided
that $B\subseteq A$ with $B\in \cI$ implies that
\[ \gamma(B)=0 \sptext{1}{or}{1} \gamma(A\setminus B)=0. \]
\end{defin}

\begin{lemma}\label{lemma:quasi_atoms}
Let $(S,\cI,\gamma)$ be a $\sigma$-finite measure space with $\gamma(S)>0$ and $A,A_1,A_2$ be quasi-atoms.
\begin{enumerate}
\item If $B\in \cI$ and $\gamma(A\Delta B)=0$, then $B$ is a quasi-atom.
\item If $A_1\subseteq A_2$, then $\gamma(A_2\setminus A_1)=0$.
\item Either $\gamma(A_1\cap A_2)=0$ or $\gamma(A_1\Delta A_2)=0$.
\item There exist countably many pairwise disjoint quasi-atoms $(A_i)_{i\in I}$ such that
      $S\setminus (\bigcup_{i\in I} A_i)$ does not contain any quasi-atom. For any 
      quasi-atom $A$ there is an $i\in I$ such that $\gamma (A\Delta A_i)=0$.
\end{enumerate}
\end{lemma}
\medskip

\begin{lemma}
Let $(S,\Sigma,\gamma)$ be a $\sigma$-finite measure space with $\gamma(S)>0$ and $\AA$ be a group of automorphisms of $S$ 
such that $(S,\nI(\AA),\gamma)$ is $\sigma$-finite. Assume that $(A_i)_{i\in I}\subseteq \nI(\AA)$ is a 
countable collection of quasi-atoms like in Lemma \ref{lemma:quasi_atoms}~(4).
Then for a function $\xi\colon\,S\arre \RR$ the following assertions are equivalent:
\begin{enumerate}
\item $\xi$ is $\conj{\nI(\AA})$-measurable.
\item There exists a $\Sigma$-measurable $\eta$ which is constant on the orbits and the quasi-atoms
      $(A_i)_{i\in I}$  and such that $\eta=\xi$ $\gamma$-a.e.
\end{enumerate}
\end{lemma}

\begin{proof}
$(2)\Longrightarrow (1)$ Using Lemma \ref{lem:equiinv} we get that $\eta$ is $\nI(\AA)$-measurable, so that $\xi$ is 
$\conj{\nI(\AA})$-measurable.\\
$(1)\Longrightarrow (2)$ First we find an $\xi_0\in [\xi]$ that is $\nI(\AA)$-measurable.  It can be easily seen that 
$\xi_0$ can be modified to an $\nI(\AA)$-measurable random variable $\eta$ satisfying the 
claimed properties.
\end{proof}


\section{Some technical proofs}
\label{section:technical_proofs}

\begin{lemma}\label{lemma:composition_and_Tg}
Let $g\in\mpmd$, $F_1,\ldots,F_n \in L_2(\nF^X)$ and
$f\colon \R^n\to \R$ be continuous such that
$f(F_1,\ldots,F_n)\in L_2(\cF^X)$. Then
$T_g f(F_1,\ldots,F_n) = f(T_gF_1,\ldots,T_gF_n)$ a.s.
\end{lemma}

\begin{proof}
As in the proof of Lemma \ref{lemma:permpoisson} it is enough to prove that 
\[
    T_g \psi_L (f(F_1,\ldots,F_n)) \\
  = \psi_L (f (T_g F_1,\ldots,T_g F_n) ) \as \]
so that we can assume that $f\in \nC_b(\R^n)$.
By Lemma \ref{lemma:HX0dense}, we find $\nH^X \ni F_{i,k}\to F_i$ in $L_2(\cF^X)$ as $k\to \infty$. By a diagonal argument, 
we find a sub-sequence $(k_l)_{l=1}^\infty$ such that, for $l\to \infty$,
$F_{i,k_l} \to F_i$ a.s.\ and
$T_g F_{i,k_l} \to T_g F_i$ a.s.\ for $i=1,\ldots,n$.
Therefore, as $l\to \infty$, 
\[ f(F_{1,k_l},\ldots,F_{n,k_l}) \to f(F_1,\ldots,F_n)  \sptext{1}{and}{1} 
   f(T_g F_{1,k_l},\ldots,T_g F_{n,k_l}) \to f(T_g F_1,\ldots,T_g F_n) \]
a.s.\ and therefore, by the boundedness of $f$, we have convergence in $L_2(\cF^X)$.
We conclude by
\begin{align*}
      T_g f(F_1,\ldots,F_n)
  &=  \lim_{l\to \infty} T_g f(F_{1,k_l},\ldots,F_{n,k_l}) \\
  &=  \lim_{l\to \infty} f(T_g F_{1,k_l},\ldots,T_g F_{n,k_l}) 
  =   f(T_g F_1,\ldots,T_g F_n),
\end{align*}
where the limits are taken in $L_2(\cF^X)$.
\end{proof}
\smallskip

\begin{proof}[Proof of Lemma \ref{lemma:permbrownian}]
From the L\'evy-It\^o decomposition \cite[Theorem 19.2]{sato:99} we know that there is a set 
$\Omega_0$ of measure one and a sequence $(\alpha_N)_{N=2}^\infty \subseteq \R$, such that for 
all $\omega\in \Omega_0$, $r\in [0,1]$, and $E_N:=(-N,-\frac{1}{N})\cup(\frac{1}{N} ,N)$, one has 
\[ 
      \s B_r(\omega)
   =  X_r(\omega)  - \lim_{\genfrac{}{}{0pt}{}{N \to \infty}{N\ge 2}} 
      \left [ \left (\int_{(0,r]\times E_N} x \dd N(s,x)\right )(\omega) - \alpha_N r
             \right ].
\]
Using the truncations $\psi_L$, $L\in \NN$, we get therefore
\equa
\s B_t & = & \lim_{L\to \infty} \psi_L \left ( X_t  - \lim_{\genfrac{}{}{0pt}{}{N \to \infty}{N\ge 2}} 
      \left [ \left (\int_{(0,t]\times E_N} x \dd N(s,x)\right )- \alpha_N t 
             \right ] \right ) \as, \\
\s B_{g((0,t])} & = & \lim_{L\to \infty} \psi_L \left ( X_{g((0,t])}  - \lim_{\genfrac{}{}{0pt}{}{N \to \infty}{N\ge 2}} 
      \left [ \left (\int_{g((0,t])\times E_N} x \dd N(s,x)\right )- \alpha_N t 
             \right ] \right ) \as,
\tion
where we assume that $g$ is represented by some fixed permutation of dyadic intervals and 
$B_{g((0,t])}$ and $X_{g((0,t])}$ are obtained by finite differences over these intervals in the canonical way. Moreover, 
the term $\alpha_N t$ in the second equation appears due to the fact that $g$ is measure preserving.
Therefore, it is sufficient to prove that
\begin{multline*} T_g \psi_L \left ( X_t  - \lim_{\genfrac{}{}{0pt}{}{N \to \infty}{N\ge 2}} 
      \left [ \left (\int_{(0,t]\times E_N} x \dd N(s,x)\right )- \alpha_N t 
             \right ] \right ) \\ = \psi_L \left ( X_{g((0,t])}  - \lim_{\genfrac{}{}{0pt}{}{N \to \infty}{N\ge 2}} 
      \left [ \left (\int_{g((0,t])\times E_N} x \dd N(s,x)\right )- \alpha_N t 
             \right ] \right ) \as
 \end{multline*}
Because of the almost sure convergence in $N\to\infty$ it is 
 sufficient to verify that
\begin{multline*} 
       T_g \psi_L \left ( X_t  - \int_{(0,t]\times E_N} x \dd N(s,x) + \alpha_N t \right ) \\ 
     = \psi_L \left ( X_{g((0,t])}  - \int_{g((0,t])\times E_N} x \dd N(s,x) + \alpha_N t  \right ) \as
\end{multline*}
for $N\ge 2$, or
\begin{multline*} 
     T_g \psi_L \left ( \psi_K(X_t)  - \int_{(0,t]\times E_N} x \dd N(s,x) + \alpha_N t \right ) \\ 
   = \psi_L \left ( \psi_K(X_{g((0,t])})  - \int_{g((0,t])\times E_N} x \dd N(s,x) + \alpha_N t \right ) \as
 \end{multline*}
 for $K,L\in \NN$. As the integral terms belong to $L_2(\cF^X)$, this follows from  
 Lemmas \ref{lemma:permpoisson} and \ref{lemma:composition_and_Tg}.
\end{proof}
\smallskip

\begin{proof}[Proof of Lemma \ref{lemma:borelstable}]
\hspace*{-.4em} A Borel measurable function $\f$ can be approximated by truncation by 
bounded Borel measurable functions $\f_L:= \psi_L(\f)$, $L\in\NN$,
and $\f_L(F_1,\ldots,F_n)\in \bH_\G$ implies $\f(F_1,\ldots,F_n)\in \bH_\G$ by monotone convergence and the completeness of $\bH_\G$
(which is easy to check as the operators $T_g\colon L_2(\nF^X)\to L_2(\nF^X)$ are isometries).
Assuming that $\f$ is bounded, we approximate $\f$ point-wise by simple functions $\f_k$ with $\|\f_k\|_\infty\leq \|\f\|_\infty$. 
It follows that $\f_k(F_1,\ldots,F_n)\arre \f(F_1,\ldots,F_n)$ in $L_2(\nF^X)$ by dominated convergence. 
Therefore, it is sufficient to check the statement for $\f=\cf_B$ where $B$ is a Borel set from $\R^n$.
Using the outer regularity of the law of $(F_1,\ldots,F_n)$ we can verify this in turn by using $\f\in\nC_b(\RR^n)$. But this case follows from Lemma~\ref{lemma:composition_and_Tg}.
\end{proof}

\end{appendix}

\paragraph{\bf Acknowledgement} We want to thank the referee and Christel Geiss for the valuable comments and 
suggestions that improved and completed this article.


\bibliographystyle{amsplain}

\begin{thebibliography}{12}

\bibitem[Applebaum(2009)]{applebaum:09}
D. Applebaum.
\newblock \emph{L\'evy Processes and Stochastic Calculus}.
\newblock Cambridge University Press, 2nd edition, 2009.
 
\bibitem[Bauer(2001)]{bauer:01}
H. Bauer.
\newblock \emph{Measure and Integration Theory}.
\newblock de Gruyter, 2001.

\bibitem[Baumgartner and Geiss(2014)]{baumgartner:geiss:14:arxiv}
F. Baumgartner and S. Geiss.
\newblock Permutation Invariant Functionals of L\'evy Processes.
\newblock ArXiv 1407.3645.

\bibitem[Billingsley(1999)]{billingsley:99}
P. Billingsley.
\newblock \emph{Convergence of Probability Measures}, 2nd edition.
\newblock John Wiley \& Sons, 1999.

\bibitem[Briand and Labart(2014)]{briand:labart:14}
P. Briand and C. Labart.
\newblock Simulation of BSDEs by Wiener chaos expansion. 
\newblock \emph{Annals of Appl. Prob.}, 24(3), 1129--1171, 2014.

\bibitem{ditella:engelbert:16}
P. Di Tella and H.-J. Engelbert.
\newblock The chaotic representation property of compensated-covariation stable families of
         martingales.
\newblock To appear in \emph{Annals of Prob.}

\bibitem{eddahbi:sole:vives:2005}
M. Eddahbi, J.L. Sol\'e, and J. Vives.
\newblock A Stroock formula for a certain class of L\'evy processes and applications to finance.
\newblock \emph{J. Appl. Math. Stoch. Anal.}, 3, 211--235, 2005. 

\bibitem[Geiss et~al.(2010)Geiss, Geiss, and Laukkarinen]{geiss:geiss:laukkarinen:13}
C. Geiss, S. Geiss, and E. Laukkarinen.
\newblock A note on Malliavin fractional smoothness for L\'evy processes and approximation.
\newblock \emph{Potential Anal.}, 39(3), 203--230, 2013.

\bibitem{cgeiss:labart:15}
C. Geiss and C. Labart.
\newblock Simulation of BSDEs with jumps by Wiener chaos expansion.
\newblock To appear in \emph{Stoch. Proc. Appl.}

 \bibitem[Geiss and Laukkarinen(2011)]{geiss:laukkarinen:11}
C. Geiss and E. Laukkarinen.
\newblock Denseness of certain smooth L\'evy functionals.
\newblock \emph{Probab. Math. Statist.}, 31(1), 1--15, 2011.
 
\bibitem[Geiss and Steinicke(2013)]{geiss:steinicke:1212.3420v3}
C. Geiss and A. Steinicke.
\newblock {$L_2$}-variation of L\'evy driven BSDEs with non-smooth terminal conditions.
\newblock ArXiv 1212.3420v3.

\bibitem[Geiss and Steinicke(2013)]{geiss:steinicke:13}
C. Geiss and A. Steinicke.
\newblock {$L_2$}-variation of L\'evy driven BSDEs with non-smooth terminal conditions.
\newblock \emph{Bernoulli}, 22(2), 995--1025, 2016.

\bibitem[It\^o(1956)]{ito:56}
K. It\^o.
\newblock Spectral type of the shift transformation of differential processes
          with stationary increments.
\newblock \emph{Trans. AMS}, 81(2), 253--263, 1956.
 
\bibitem[Kallenberg(2002)]{kallenberg:02}
O. Kallenberg.
\newblock \emph{Foundations of Modern Probability}.
\newblock Springer, 2nd edition, 2002.
 
\bibitem[Last and Penrose(2011)]{last:penrose:11}
G. Last and M.D. Penrose.
\newblock Poisson process Fock space representation, chaos expansion and covariance inequalities.
\newblock \emph{Probab. Theory Relat. Fields,} 150, 663--690, 2011.

\bibitem[Lee and Shih(2004)]{lee:shih:04}
Y.-J. Lee and H.-H. Shih.
\newblock The product formula of multiple L\'evy-It\^o integrals.
\newblock \emph{Bull. of the Institute of Mathematics, Academia Sinica,} 32, 71--95, 2004.

\bibitem[Nualart and Schoutens(2000)]{nualart:schoutens:00}
D. Nualart and W. Schoutens.
\newblock Chaotic and predictable representations for L\'evy processes.
\newblock \emph{Stoch. Proc. Appl.,} 90, 109--122, 2000. 

\bibitem[Giovanni Peccati and Murad S.\ Taqqu (2011)]{peccati:taqqu:11} 
G. Peccati, M. Taqqu.
\newblock \emph{Wiener Chaos: Moments, Cumulants and Diagrams.}
\newblock Springer, 2011.

\bibitem[Privault(2009)]{privault:09}
N. Privault.
\newblock \emph{Stochastic Analysis in Discrete and Continuous Settings.}
\newblock Lecture Notes in Mathematics 1982, Springer, 2009.

\bibitem[Protter(2004)]{protter:04}
P. Protter.
\newblock \emph{Stochastic Integration and Differential Equations.}
\newblock Springer, 2004.
 
\bibitem[Sato(1999)]{sato:99}
K. Sato.
\newblock \emph{L\'evy Processes and Infinitely Divisible Distributions}.
\newblock Cambridge University Press, 1999.
 
\bibitem[Shiryaev(1996)]{shiryaev:96}
A.~N. Shiryaev.
\newblock \emph{Probability}.
\newblock Springer, 2nd edition, 1996.

\bibitem[Sole and Utzet(2008)]{sole:utzet:08}
J. Sol\'e and F. Utzet.
\newblock On the orthogonal polynomials associated with a L\'evy process.
\newblock \emph{Annals of Prob.}, 36(2), 765--795, 2008.

\bibitem[Sole, Utzet, and Vives(2007)]{sole:utzet:vives:07}
J. Sol\'e, F. Utzet, and J.Vives.
\newblock Chaos expansions and Malliavin calculus for L\'evy processes.
\newblock Stochastic Analysis and Applications,
          Abel Symp. 2, 595--612, 2007.

\bibitem{tang:li:94}
S. Tang and X. Li.
\newblock Necessary conditions for optimal control of stochastic systems with random 
          jumps.
\newblock \emph{SIAM J. Control and Optim.,} 32, 1447--1475, 1994.

\bibitem[Yan Yip, Stephens and Olhede(2010)]{yip:stephens:olhede:10} 
W. Yan Yip, D. Stephens, and S. Olhede.
\newblock The explicit representation of the powers of increments of L\'evy processes.
\newblock \emph{Stochastics,} 82, 257--290, 2010.

\end{thebibliography}

\end{document}